\documentclass[leqno]{article} 
\usepackage{amsmath,amsthm}
\usepackage{amssymb,amscd,latexsym}
\usepackage{boxedminipage}
\usepackage{eepic}
\usepackage{epic}
\usepackage{wrapfig}
\usepackage{graphicx}

\theoremstyle{plain}
\newtheorem{thm}{Theorem}[section]

\newtheorem*{proposition}{Proposition}

\newtheorem*{theo}{Theorem}
\newtheorem*{def-theo}{Definition-Theorem}

\newtheorem*{corollary}{Corollary}
\newtheorem{lem}[thm]{Lemma}

\theoremstyle{definition}

\newtheorem*{definition}{Definition}

\newtheorem{defn}[thm]{Definition}

\newtheorem{rem}[thm]{Remark}

\numberwithin{equation}{section}
\newcommand{\bp}{\begin{pmatrix}}
\newcommand{\ep}{\end{pmatrix}}
\newcommand{\bps}{\begin{smallmatrix}}
\newcommand{\eps}{\end{smallmatrix}}

\def\R{{\mathbb R}}
\def\Z{{\mathbb Z}}
\def\A{{\mathcal A}}
\def\B{{\mathcal B}}

\def\al{{\alpha}}

\def \0{{\bf 0}}
\def \1{{\bf 1}}

\def\Tr{\mathrm{Tr}}

\def \rank{\mathrm{rank}}

\begin{document}

\title{The Zero Locus  of  the $F$-triangle } 
 \author{ Kyoji Saito
 \date{}
\footnote{Kavli\! Institute\! for\! the\! Physics\! and\! Mathematics\! of\! the\! Universe\! (WPI),\! 
the\! University\! of\! Tokyo}}

\maketitle

{\renewcommand{\baselinestretch}{0.1}}

\begin{abstract} 
We are interested in the zero locus of a Chapoton's $F$-triangle as a polynomial in two real variables $x$ and $y$. An expectation is that (1)  {\it the $F$-triangle of rank $l$ as a polynomial in $x$ for each fixed $y\in[0,1]$, has exactly $l$ distinct real roots in $[0,1]$}, and (2) {\it $i$-th root $x_i(y)$ {\rm($1 \!\le \! i \! \le \! l$)} as a function on $y \! \in\! [0,1]$ is monotone decreasing}. In order to understand these phenomena, we slightly generalized the concept of $F$-triangles and study the problem on the space of such generalized triangles. We analyze the case of low rank in details and show that the above expectation is true. We formulate inductive conjectures and questions for further rank cases.\ 
This study gives a new insight on the zero loci of $f^+$\!- and $f$-polynomials.
%
\end{abstract}


\section{Chapoton's $F$-triangle}

We recall the {\it $F$-triangle} $F_\Phi$  by F. Chapoton(\cite{C1}) associated with a complete simplicial fan $\Delta(\Phi)$, called a {\it cluster fan}, associated with a finite root system $\Phi$,  and introduced by Fomin and Zelevinsky
 \cite{F-Z1,F-Z2}. \footnote
{In order to adjust to the study in the present note, the definition \eqref{ftriangle} changes its sign of variables from the original definition of Chapoton. One should be cautious that this change causes several sign changes in the sequel (e.g. \eqref{partialderivative}, Proposition 4., \eqref{D-derivation}, \ldots etc.). }.

Let $\Phi$ be a finite (not necessary irreducible) root system of rank $l$, 
  $\Phi^+ \subset \Phi$  be a positive root system, and $\Pi=\{\alpha_i\}_{i\in I} \subset \Phi^+$ be the associated simple basis where $I$ is the index set of order $l$ and we identify it with associated Dynkin diagram whose underlying set is $I$. 
Let $\Phi(J)$ denotes the finite root system associated with the full sub-diagram $J\subset I$.
%
 A symmetric compatiblity relation $(\al\Vert \beta)=0$ for  $\al,\beta \in \Phi_{\ge-1}:=\Phi^+ \cup \{-\Pi \}$ was introduced in \cite{F-Z3}:

\begin{theo} {\rm (\cite{F-Z3} Theorem 1.10)}
The cones spanned by subsets of mutually compatible elements in $\Phi_{\ge-1}$ define a complete simplicial fan $\Delta(\Phi)$.
\end{theo}

Taking the distinction between positive and negative vertices $-\Pi$ of $\Delta(\Phi)$ in account, F. Chapoton (\cite{C1} (1)) introduced a refinement of the face counting generating function, called the $F$-triangle (see Footnote 1).
\begin{defn}
\ For a root system $\Phi$,   the $F$-triangle is 
 \begin{equation}
 \label{ftriangle}
 \begin{array}{rcl}
F(\Phi)& =& F_\Phi(x,y) \ =\ \sum_{k=0}^l \sum_{m=0}^l f_{k,m} (-x)^k (-y)^m, \\
\end{array}
\end{equation}
where $f_{k,m}$ is the cardinality of the set of simplicial cones of $\Delta(\Phi)$ spanned by exactly $k$ positive roots and $m$ negative simple roots. The coefficient $f_{k,m}$ vanishes if $k+m>l$, hence the name triangle. 
\end{defn}

 For a formal convenience, we include an ``empty root system'' $\Phi(\emptyset)$ in the discussion.  In that case, we set $F(\Phi(\emptyset)): =  1$. The next and the most basic $F$-triangle is the case of rank 1, which is of type $A_1$ and is given by 
$$
F_{A_1}= 1-x-y.
$$
In general, we have the following reductions.

\begin{proposition} {\bf 1}
\label{chapoton3} 
{\rm (\cite{C1} Proposition 3)}
The F-triangle has the following properties.

{\bf 1.}  If $\Phi$ and $\Phi'$ are two root systems, one has $F(\Phi\times\Phi')=F(\Phi)\times F(\Phi')$.\footnote
{If $\Phi_1, \Phi_2$ are root systems in euclidean spaces $V_1,V_2$, respectively, then $\Phi_1\times \Phi_2 := \{\Phi_1\times\{0\}\} \cup \{\{0\}\times\Phi_2\}$ in $V_1\times V_2$, whose set of simple roots is the disjoint union $I_1\cup I_2$. }

{\bf 2.} If $\Phi$ is an irreducible root system on a Dynkin diagram $I$, then one has
\begin{equation}
\label{partialderivative}
\partial_y F(\Phi(I)) = -\sum_{i\in I} F(\Phi(I\setminus \{i\})).
\end{equation}
\end{proposition}

Actually, combining 1. and 2. of Proposition, \eqref{partialderivative} is valid for non-irreducible root systems. E.g., suppose the diagram $I$ decomposes into a disjoint union of two connected diagrams $I_1$ and $I_2$ so that $\Phi=\Phi(I)=\Phi(I_1)\times\Phi(I_2)$.
Then 
\[
\begin{array}{rll}
\partial_yF(\Phi(I))= &\partial_y(F(\Phi(I_1))F(\Phi(I_2))) \\
= &(\partial_yF(\Phi(I_1)))F(\Phi(I_2)) +F(\Phi(I_1))(\partial_yF(\Phi(I_2)) )\\
= &\!\!\!\!-(\sum_{i\in I_1} \! F(\Phi(I_1\!\! \setminus \!\! \{i\}))) F(\Phi(I_2)) \! - \! F(\Phi(I_1))(\sum_{i\in I_2} \! F (\Phi(I_2 \!\! \setminus \!\! \{i\}))) \\
= &\!\!\!\!-\sum_{i\in I_1}F(\Phi((I_1\setminus\{i\})\cup I_2)) - \sum_{i\in I_2} F (\Phi(I_1 
\cup (I_2\setminus\{i\}))) \\
= &\!\!\!\! - \sum_{i\in I=I_1\cup I_2} F(\Phi(I\setminus\{i\}))
\end{array}
\]
Chapoton also has shown the following rotational symmetry of $F$-triangles.

\begin{proposition} {\bf 2}
\label{chapoton5}
 {\rm (\cite{C1} Proposition 5)} 
 Let $\Phi$ be a root system of rank $l$, then one has 
\begin{equation}
\label{rotation}
F_{\Phi}(x,y) \quad = \quad (-1)^l\ F_{\Phi}(1-x,1-y).
\end{equation}
\end{proposition}

\begin{rem}
\label{noncrystal}
The construction of the cluster fan is extended to any finite Coxeter group by the authors \cite{}, which we symbolically denote $\Delta_\Phi$ (where $\Phi$ is the set of reflections and $I$ is a simple generator system of the Coxeter group).  Hence, the $F$-triangle $F(\Phi)$ is defined also for any ``non-crystallographic finite root system" $\Phi$. Then, Propositions 1 and 2 hold for extended $F$-triangles. 

Explicit determination of $F$-triangles are given by Chapoton \cite{C1} and others (see, for instance, \cite{Ar,K}). In the present paper, we shall freely use the results without referring to them explicitly. 
\end{rem}

\noindent
{\bf Other Examples.}  In a personal communication \cite{C2} to the author, F. Chapoton informed that there exists a one parameter (interpolation by the Coxeter number $h$) family of $F$-triangles of rank 3.
In the present note, we also study them under the name Chapoton family $F_{Chap(h)}$ (see \S4 Table and \S5 3. Rank 3 case, d) Figure 8).



\section{Polynomials $\Z_{\ge0}[Cox]$}

Inspired by the descriptions in \S1, we introduce the set $\Z_{\ge0}[Cox]$ of polynomials generated by isomorphism classes of finite Coxeter groups,  and introduce a generalization of $F$-triangles on that set, which we call again $F$-triangles.  

Let us consider the set $Cox:=\{$all isomorphism classes of finite reflection groups (not necessary crystallographic)$\}$. It has a natural multiplicative monoid structure, i.e.\  the product, denoted by $\times$ but which is usually omitted, of two classes of Coxeter groups means the isomorphism class of the product group. We set
\vspace{-0.3cm}
$$
\Z_{\ge0}[Cox]:=\{\text{ $\Z_{\ge0}$-linear combinations of elements of $Cox$}\},
\vspace{-0.1cm}
$$
carrying the additive and multiplicative structures, 
where the sum is formal sum of symbols without any geometric meaning, and the product is the distributive extension of the product defined on the monoid $Cox$.
That is, $\Z_{\ge0}[Cox]$ is the part, consisting of positive coefficients elements, of the polynomial ring $\Z[Cox]$ whose  generators are isomorphism classes of  irreducible   Coxeter groups: $A_l \ (l\in\Z_{\ge0}), B_l \ (l\in \Z_{\ge2}),$ $ D_l \ (l\in\Z_{\ge4}), E_6, E_7, E_8, F_4, G_2, H_3, H_4, I_2(p) \ (p\in\Z_{\ge2})$, where we identify $A_1^2=I_2(2),$ $A_2=I_2(3), B_2=I_2(4), G_2=I_2(6)$.
\footnote
{For simplicity, we denote the formal sum $\Phi\oplus \Phi'$ by $\Phi+\Phi'$. One should not confuse $\Phi+\Phi'$ with $\Phi \times \Phi'$. We stress that $\Phi+\Phi'$ merely means a symbolic sum without any geometric operation on the Coxeter groups.
}

\smallskip
The 
$\Z_{\ge0}[Cox]$ is graded in the following sense. For a monomial $\Phi\in Cox$, we set $\rank(\Phi):=$the number of simple reflections of  $\Phi$. We have an obvious addition rule: $\rank(\Phi_1\Phi_2)=\rank(\Phi_1)+\rank(\Phi_2)$. 
We call a polynomial $\Phi\in \Z_{\ge0}[Cox]$ {\it homogeneous} of $\rank(\Phi)=l$, if all its summand  monomials have the same rank $l$.
We set 
\vspace{-0.2cm}
$$
\Z_{\ge0}[Cox]_l:=\{\Phi \in \Z_{\ge0}[Cox]\mid \Phi \text{ is homogenous of rank } l\} 
\vspace{-0.1cm}
$$
for $l\in\Z_{>0}$ and $\Z_{\ge0}[Cox]_0=\Z_{\ge0}.$

\medskip
We now introduce  three operations on $\Z_{\ge0}[Cox]$: trace $Tr$, boundary map $\partial$, and $F$-triangle map $F$.

\smallskip
\noindent
{\bf 1) Trace morphism.} 
\vspace{-0.2cm}
\[
\begin{array}{c}
Tr \ : \ \Z_{\ge0}[Cox] \longrightarrow \Z_{\ge0} \ ,  \quad \Phi=\oplus_i \ (a_i \Phi_i) \ \ \mapsto\ \ \sum_i a_i
\end{array}
\vspace{-0.1cm}
\]
where $\Phi_i\in Cox$ is a monomial in $\Z_{\ge0}[Cox]$. 
Clearly, the trace $\ Tr$ is a map which preserves both addition and product.

\smallskip
\noindent
{\bf 2)  Boundary map.}  

Let us introduce a boundary map $\partial: \Z_{\ge0}[Cox]_l\to \Z_{\ge0}[Cox]_{l-1}$ for $l\in\Z_{>0}$. Namely, for any  $\Phi=\Phi(I) \in Cox$ (where $I$ is  the Coxeter diagram for $\Phi$), we set 
\vspace{-0.2cm}
\[
\partial \ \Phi \ := \ \oplus_{i\in I} \ \Phi(I\setminus\{i\}).
\vspace{-0.1cm}
\]
Then, we extend  the action $\partial$ to the whole $\Z_{\ge0}[Cox]$ additively.

\begin{proposition}
{\bf 3.} i)  The boundary map $\partial$ satisfies the Leibniz rule.

ii) If $\Phi\in \Z_{\ge0}[Cox]_l$, then $Tr(\partial\Phi)=l \cdot Tr(\Phi)$.
\end{proposition}
\begin{proof} It is sufficient to show the case of a product of two Coxeter groups $\Phi(I_1\cup I_2)=\Phi(I_1)\Phi(I_2)$  where the diagram $I_1 \cup I_2$ is a disjoint union of two sub-diagrams $I_1$ and $I_2$. Then, we trivially have 
\[
\begin{array}{rll}
\vspace{0.1cm}
&\partial\Phi(I))=  \sum_{i\in I=I_1\cup I_2} \Phi(I\setminus\{i\})\\
\vspace{0.1cm}
= &\sum_{i\in I_1}\Phi(I_1\setminus\{i\})\Phi(I_2) + \sum_{i\in I_2} \Phi(I_1) \Phi(I_2\setminus\{i\}) \\
= &(\partial \Phi(I_1))\Phi(I_2) +\Phi(I_1)(\partial\Phi(I_2) )\\
\end{array}
\vspace{-0.5cm}
\]
\end{proof}
\noindent
For $\Phi\in \Z_{\ge0}[Cox]_l$, we have the relation:
$
Tr(\partial(\Phi))\quad =\quad  l \cdot Tr(\Phi) \ .
$


\medskip
\noindent
{\bf 3) $F$-triangle.} 

We re-introduce the $F$-triangle map on $\Z_{\ge0}[Cox]$ as the additive extension of the original $F$-triangle: 
\vspace{-0.1cm}
\begin{equation}
\label{generalftriangle}
F\ : \ \Z_{\ge0}[Cox] \quad  \longrightarrow \quad \Z[x,y], \qquad \oplus_i\Phi_i \mapsto \sum_i F(\Phi_i)
\vspace{-0.3cm}
\end{equation}
where each term $F(\Phi_i)$ is the original $F$-triangle introduced  in \S1.
The map $F$ is not only additive but also multiplicative due to Proposition \ref{chapoton3}. 

 Let us denote by $F_{\Phi}(x,y)$ the image polynomial of $\Phi\in \Z_{\ge0}[Cox]$ by the map $F$ and call it again  the $F$-triangle associated with $\Phi$. We see that the map $F$ factors the trace map as $F_\Phi(0,0)=Tr(\Phi)$.  One should also be cautious that the $F$-triangle $F_\Phi$ associated with a homogeneous element $\Phi \in \Z_{\ge0}[Cox]_l$ is not a homogeneous polynomial in $\Z[x,y]$ in the usual sense.

\begin{rem}
Compared with the definition of $F$-triangle in \S1, present definition is extended in two way. 1) Non-crystallographic root systems, i.e.\ of types $H_3$, $H_4$ and $I_2(p)$, are included (it was already mentioned in Remark 1.2), and 2) Positive linear combinations of triangles are included.
These extensions of the class of functions, allowing to sum up polynomials for different root systems and closed under the derivation, is a minor change, however, it is necessary for a formulation in the induction step and, also presumably, for the solution of conjectures in the following sections (c.f.\ Remark 4.2 and Remark 6.*) .
\end{rem}

Propositions 1 and 2 
are valid for the extended $F$-triangles. 
In particular, 

\begin{proposition} {\bf 4.} 
The $F$-triangle map $F$ commutes with $\partial$ on $\Z_{\ge0}[Cox]$ and $-\partial_y$ on $\Z[x,y]$. That is, one has \ $\partial_y F_\Phi= -F_{\partial \Phi}$ for $\Phi\in \Z_{\ge0}[Cox]$.
\end{proposition}

\medskip
Let us give another important feature of $F$-triangles, which reflects that the cluster cone contains just one $l$-dimensional simplicial cone of negative vertices. 
We shall use this formula 
in the counting of Sturm roots in \S6 Discusssions 12.

\begin{proposition} {\bf 5.}
Let $\Phi\in \Z_{\ge0}[Cox]$ be a homogeneous element of rank $l$. Then 
\begin{equation}
\label{x=01}
F_{\Phi}(1,y)=\ Tr(\Phi) \ (-y)^l  \quad \text{and} \quad F_{\Phi}(0,y)=\Tr(\Phi) \ (1-y)^l.
\end{equation}
\end{proposition}
\begin{proof} It is sufficient to prove the formula for each direct summand  of $\Phi$ since both hand sides are linear in $\Phi$. Let us first show the second formula. By definition, $F_{\Phi}(0,y)$ is the generating function of cones contained in the cone spanned by $-\Pi$, which is a simplicial cone of dimension $l$, and hence is equal to $\sum_{i=0}^l (-1)^i(\substack{l\\ i})y^i=(1-y)^l$.  Then, the first formula is obtained from this by applying the rotational symmetry \eqref{rotation}.
\end{proof}


\section{ Polyhedral cone $P_l$ of $F$-triangles of rank $l$ }

By extending the coefficients from $\Z_{\ge0}$ to $\R_{\ge0}$, we consider the infinite dimensional  convex cone $\R_{\ge0}[Cox]:=\Z_{\ge0}[Cox]\otimes_{\Z_{\ge0}}\R_{\ge0}$ which decomposes into the direct product 
of cones $\R_{\ge0}[Cox]_{l}:=\Z_{\ge0}[Cox]_{l}\otimes_{\Z_{\ge0}} \R_{\ge0}$ for $l\in\Z_{\ge0}$.\footnote
{We mean by a ``cone" simply a set which is invariant under the multiplication of $\R_{\ge0}$.
}
The maps $Tr$, $D$ and $F$ \eqref{generalftriangle} extends $\R$-linearly to those $\R$-cones to the $\R$-extensions of target spaces. We shall call the image $F(\Phi)$ for $\Phi\in \R_{\ge0}[Cox]_l$ again a $F$-triangle. 
Although the cone $\R_{\ge0}[Cox]_{l}$ for each rank $l\in \Z_{\ge0}$ is infinite dimensional (except for the case $l=1$),  we show that its image: the space of $F$-triangles of rank $l$ 
\begin{equation}
\label{polygon}
P_l \ := \ \{F(\Phi)\mid \Phi\in \R_{\ge0}[Cox]_{l} \ \& \ \Tr(\Phi)=1\}
\end{equation} 
is a {\it finite dimensional semi-algebraic convex polyhedron}\footnote
{We mean that $P_l$ is a convex semi-algebraic set in a finite dimensional $\R$-affine space whose closure is a polyhedron in usual sense, i.e.\ it is a finite intersection of closed half spaces, however $P_l$ itself may not be closed and some of facets of the polygon may be missing. 
}, but  which is un-bounded and non-closed except for the case $l=1$:
\begin{equation}
\label{p1}
P_1\quad =\quad \{F_{A_1}:=1-x-y\}.
\end{equation}
 The reason for these facts comes from the fact that $F$-triangles of type $I_2(p)$ for $p\in\Z_{\ge 2}$ are strongly algebraically dependent each other, and this fact leads to an introduction of  $F$-triangles of ``virtual" type $I_{2k}(s_1,\cdots,s_k)$ for $k\in\Z_{>0}$ with only finite number of non-negative real parameters $s_1,\cdots,s_k\in\R_{\ge0}$.
 
\bigskip 

\noindent
{\bf I.}
 We first restrict our attention to the subset of $\Z_{\ge0}[Cox]$ generated by rank 2 Coxeter groups  $I_2(p)$ ($p=2,3,4,\cdots$), where we recall that $I_2(2)=A_1^2$ is reducible and
$
F_{I_2(2)}=F_{A_1^2}=F_{A_1}^2= (1-x-y)^2 
 $. 

\begin{lem} 
\label{I_2}
Let us introduce a ``virtual infinity" polynomial
\begin{equation}
\label{Finfty}
F_\infty:= x(x-1) .
\end{equation}
Then, $F$-triangles of rank 2 Coxeter groups  are given by
\begin{equation}
\label{I2}
F_{I_2(p)}=F_{A_1}^2 + (p-2)F_\infty \qquad \text{for \ $p\in \Z_{\ge2}$}. 
\end{equation}
\end{lem}
\begin{proof} This follows from the explicit formula:
$$
\begin{array}{rcl}
F_{I_2(p)}& =&1-2y-px+y^2+2xy + (p-1)x^2\\ 
&=& (1-x-y)^2+ (p-2)x(x-1).
\end{array}
\vspace{-0.7cm}
$$
\end{proof}

Let us introduce the  $F$-triangle of {\it virtual type} $I_2(p)$ for $p\in \R_{\ge2}$ by the same formula \eqref{I2} by extending the domain of $p$ from $\Z_{\ge2}$ to $\R_{\ge2}$.  Then, we obtain the following elementary but non-trivial fact.
 
\begin{corollary}
The set $P_2$ of all $F$-triangles of rank 2 with trace 1 is equal to the half line of all $F$-triangles of virtual type $I_2(p)$ for all $p\in \R_{\ge2}$.  That is, 
\begin{equation}
\begin{array}{rcl}
\label{p2}
P_2 & = & \{ F_{I_2(p)}=F_{I_2(2)}+sF_\infty \mid  s=p-2\in \R_{\ge0}\} \\
\\
& = & \overset{I_2(2)}{\circ} \!\!\!\!\!\!-\!\!\!-\!\!\!\!-\!\!\!-\!\!\!\!-\!\!\!\!\!\!\overset{I_2(3)}{\circ}\!\!\!\!\!-\!\!\!\!\! -\!\!\!-\!\!-\!\!\!\!-\!\!\!\!\!\!\overset{I_2(4)}{\circ}\!\!\!\!\!-\!\!\!\!\! -\!\!\!-\!\!-\!\!\!\!\!-\!\!\!\!\!\overset{I_2(5)}{\circ}\!\!\!\!\!-\!\!\!\!\!-\!\!\!-\!\!-\!\!\!\!\!-\!\!\!\!\!\overset{I_2(6)}{\circ}\!\!\!\!\!-\!\!\!\!\!-\!\!\!-\!\!-\!\! \circ \!\!-\!\!\!-\!\!- \cdots .
\end{array}
\end{equation}
\end{corollary}
\begin{proof}  This is equivalent to that the set of all $F$-triangles of rank 2 (= the set of  non-nenegative linear span of $F_{I_2(p)}$ ($p=2,3,4,\cdots$))  is equal to the set  
$$
F(\Z_{\ge0}[Cox]_{\R_{\ge0},2})= \{ t \big(F_{A_1}^2 + s F_\infty  \big)=t F_{I_2(s+2)} \mid t\in \R_{\ge0}, \ s\in \R_{\ge0} \} \ .
$$
RHS is a convex cone in the sense that any non-negative $\R$-linear combination of its elements belongs to itself. Due to {\it Lemma} \ref{I_2}, any $F_{I_2(p)}$ ($p\in\Z_{\ge2}$) is contained in this set. So, their non-negative linear combinations, i.e.~all $F$-triangles of rank 2 in LHS, are contained in RHS. 

Oppositely, let us show that any element $tF_{I_2(p)}$ for $t, s=p-2 \in \R_{\ge0}$ in RHS belongs to LHS.  If $t=0$, there is nothing to prove, and we assume $t\not=0$. Let $p_0$ be an integer such that $s+2<p_0$.  Then, we have 
$t \big(F_{A_1}^2 +sF_\infty\big)= t\frac{s}{p_0-2} F_{I_2(p_0)}+ t\frac{p_0-s-2}{p_0-2}F_{A_1}^2$ which belongs to LHS.
\end{proof}

\noindent
{\bf II.}  
\ {\bf Definition. }  Inspired by {\bf I}, for any $l\in\Z_{\ge1}$,  let us introduce $F$-triangles of {\it virtual type} $I_{l}(s_1,\cdots,s_{[l/2]})$ of rank $l$ by \footnote
{The use of the terminology ``$F$-triangle" is justified in the following Lemma \ref{Il}.}
\begin{equation}
\label{virtual}
\begin{array}{c}
 F_{I_{l}(s_1,\cdots,s_{[l/2]})}:=F_{A_1} ^{l} +  \sum_{i=1}^{[l/2]} s_i F_{A_1}^{l-2i} F_\infty^i   
\end{array}
\end{equation}
where $(s_1,\cdots,s_{[l/2]})\in \R^{[l/2]}$ are parameters belonging to the following  set\footnote
{The set $\Sigma_{[l/2]}$ carries a natural additive semi-group structure with respect to the coordinates $\underline{s}=(s_1,\cdots,s_{[l/2]})$. However, in case of $[l/2]=1$, there is an unfortunate discrepancy between the parameter $p\in \R_{\ge2}$ in the formula \eqref{I2} and the parameter $s_1$ in the formula \eqref{virtual}. Namely, they are related by the relation $p-2=s_1$.
}
\begin{equation}
\label{sigmak}
\begin{array}{c}
\Sigma_{[l/2]}:= \bigcup_{i=0}^{[l/2]}   \left\{s_1>0,\cdots,s_{i}>0, s_{i+1}=\cdots=s_{[l/2]}=0\right\}  .
\end{array}
\end{equation}
We note that the set $\Sigma_{[l/2]}$ is closed by addition. That is, it is a semi group.

By the use of those $F$-triangles of virtual types, we now describe the set 
$$
F(\R_{\ge0}[A_1, \{I_2(p)\}_{p\in\Z_{\ge2}}]_l)
$$ 
of all $F$-triangles of rank $l$ which are polynomials of $F_{A_1}$ and $F_{I_2(p)}$ ($p\in\Z_{\ge2}$). 

\begin{lem}  
\label{Il}
The set is equal to the semi-algebraic set of the cone over the $F$-triangles of virtual type $I_l$
\begin{equation}
\label{I2k}
\begin{array}{rcl}
\vspace{0.2cm}
&&F(\R_{\ge0}[A_1, \{I_2(p)\}_{p\in\Z_{\ge2}}]_l) \\
&= & \{t F_{I_{l}}(s_1,\cdots,s_{[l/2]}) \mid t\in\R_{\ge0}, (s_1,\cdots,s_{[l/2]})\in \Sigma_{[l/2]} \} 
\end{array}
\end{equation}
\end{lem}
\begin{proof}
If the rank is odd $2k+1$, the $F$-triangle in LHS should be divisible by $F_{A_1}$. So, dividing by $F_{A_1}$, we can reduce this case  to the even rank $2k$ case.
The proof for the case $l=2$ is given in {\bf I} .

Let us consider the case $l=2k$ for $k\ge 2$. Since $F_{A_1^2}=F_{A_1}^2=F_{I_2(2)}$, any element of LHS of \eqref{I2k} is a homogeneous polynomials of degree $k$ of the elements of the form \eqref{I2} 
 with non-negative coefficients.  Each monomial $F_{I_2(p_1)}\cdots F_{I_2(p_k)}$of degree $k$, i.e.~a product of $k$ elements of \eqref{I2}, belongs already to RHS set of \eqref{I2k} ({\it Proof.} We need only to take care when $s_i$'s become zero. However, it is clear that according to the number, say $i$, of factors $F_{I_2(p_j)}$ with $p_j-2=0$, we obtain $s_1, \cdots, s_{k-i}$ are positive and $s_{k-i+1}=\cdots=s_k=0$, which exactly belongs to the set $\Sigma_k$ \eqref{sigmak}).  Since $\Sigma_k$ is a convex cone,  any non-negative linear combination of such monomials is also contained in $\Sigma_k$. Thus LHS of \eqref{I2k} is contained in RHS.

The opposite inclusion relation of \eqref{I2k}  is verified by induction on $k$ as follows. Let $tF_{I_{2k}}(s_1,\cdots,s_k)$ be an element of RHS. If $s_{k}=0$, then it is divisible by $F_{I_2(2)}$ and belongs to a smaller dimensional stratum $\R_{\ge0}\times \Sigma_{k-1}$, and, so, we can apply induction hypothesis. Assume $s_{k}>0$. Let us consider a monomial $F_{I_2(p_1)}\cdots F_{I_2(p_k)}$ in LHS.  Taking $p_1,\cdots,p_k\in \Z_{\ge2}$ sufficiently large, we may assume that all coefficients of 
$\frac{s_{k}}{\prod_{i=1}^{k}(p_i-2)} F_{I_2(p_1)}\cdots F_{I_2(p_k)}=s_k \prod_{i=1}^{k}(\frac{F_{I_2(2)}}{p_i-2}+F_{\infty})$ (except for the last term $F_\infty^k$ whose coefficient is equal to $s_k$) are sufficiently small, so that the difference: 
$F_{I_{2k}}(s_1,\cdots,s_k) -\frac{s_{k}}{\prod_{i=1}^{k}(p_i-2)} F_{I_2(p_1)}\cdots F_{I_2(p_k)}$ are non-negative. In particular, the coefficient of $F_\infty^k$ in the difference is zero so that the difference is factored by $F_{I_2(2)}$ and belongs to a smaller dimensional stratum $\R_{\ge0}\times\Sigma_{k-1}$, for which we apply the induction hypothesis.
\end{proof}

\begin{rem} Note that $\Sigma_k$ is an unbounded semi-algebraic convex polyhedron in $(\R_{\ge0})^{k}$, admitting constant multiplication: $(s_1,\cdots,s_k)\mapsto (cs_1,\cdots,cs_k)$ for $c\in\R_{\ge0}$. One  should not confuse $c$ with the scaling parameter $t$ in Lemma \ref{Il}.  
As a  cone over $\Sigma_k$,  some faces of $(\R_{\ge0})^{k+1}$ is missing in $F(\R_{\ge0}[A_1, \{I_2(p)\}_{p\in\Z_{\ge2}}]_l)$.
E.g.\ if $l=2$, the boundary edge $\{ (p-2)F_\infty\mid p\in\R_{\ge2} \}$ is missing.
\end{rem}
\begin{rem}
Associated with a $(s_1,\cdots,s_k)\in \Sigma_k$ for $k\ge2$,  let us consider  the  polynomial equation $T^k+\sum_{i=1}^ks_iT^{k-i}=0$. The equation is not in general a totally real in the sense that some root of the equation may note be real numbers. This create a difficult problem to stratify the set $\Sigma_k$ according to the number of real roots. On the other hand, it is easy to show 

\smallskip
\noindent
{\bf Fact.}  {\it For any root, say $\alpha$, of the equation, we have 
 $\Re(\alpha)\le0$}. 
\end{rem}

\noindent
{\bf III.}  We return to the description of the set $P_l$.

\noindent
\begin{proposition} {\bf 5.} 
{\it The set $P_l  \subset \R_{\ge0}[x,y]$  is a finite dimensional convex semi-algebraic polyhedron. The set $F(\R_{\ge0}[Cox]_{l})$ is the cone over $P_l$.   
\begin{equation}
\label{cone}
F(\R_{\ge0}[Cox]_{l}) \quad = \quad \R_{\ge0} P_l. 
\end{equation}}
\end{proposition}
\begin{proof}  For $k\in \Z_{\ge0}$, consider the set 
$J_k$ of all  isomorphsim classes of finite Coxeter groups whose irreducible factors are of rank greater or equal than 3. Obviously, we have $\#J_k<\infty$. 
Using the set $J_k$, any $F$-triangle $F$ of rank $l$ has the  expression
$$
\begin{array}{c}
F= \sum_{k=0}^{l}  \big(\sum_{\Phi\in J_k}  F_\Phi \ H_{\Phi} \big)
\end{array}
$$ 
where $H_{\Phi} \in F(\R_{\ge0}[A_1, \{I_2(p)\}_{p\in\Z_{\ge2}}]_{l-k})$.  
This implies that
$$
F(\R_{\ge0}[Cox]_{l})= \sum_{k=0}^{l}  \big(\sum_{\Phi\in J_k} F_\Phi \ F(\R_{\ge0}[A_1, \{I_2(p)\}_{p\in\Z_{\ge2}}]_{l-k}).
\big)
$$
That is,  $F(\R_{\ge0}[Cox]_{l})$ is expressed as a finite sum of sets, where each summand set is, in view of Lemma \ref{Il}, a convex semi-algebraic polyhedral cone. Thus, $F(\R_{\ge0}[Cox]_{l})$ is also a convex semi-algebraic polyhedral cone, and its hyperplane cut $P_l$ by the equation $\Tr(F)=0$ is also a convex semi-algebraic polyhedron (unbounded).

In order to see \eqref{cone} that any ray in $F(\R_{\ge0}[Cox]_{l})$ intersects with $P_l$,  we have only to notice  $\{F_\Phi \in F(\R_{\ge0}[Cox]_{l}) \mid \Tr(F_\Phi)=0\}=\{0\}$.  
\end{proof}

In general, if a closed set  is a cone over a bounded polyhedron, then it is determined as a convex hull of the finite rays corresponding to the vertices of the polyhedron. However, $P_l$ is not bounded as we saw \eqref{p2} so that the data of its vertices is not sufficient to recover it. Let us show more precise description of the unboundedness of $P_l$.

Recall the definition \eqref{virtual} of the $F$-triangle of virtual type $I_l(\underline{s})$. We separate it into two parts, and interpret that it is  the consequence of the semi-group $\Sigma_{[l/2]}\simeq \{\sum_{i=1}^{[l/2]} s_i F_{A_1}^{l-2i} F_\infty^i \mid (\underline{s})\in \Sigma_{[l/2]}\}$ action on the $F$-triangle $F_{A_1}^l$. Then, we ask further, whether, at another point in $P_l$, the semi-group $\Sigma_{[l/2]}$ acts also? The following Lemma gives an answer to this question.

\begin{lem} 
\label{semigroupaction}
If $F\in P_l$ is an $F_{A_1^l}$-interior point of $P_l$ (that is, $F$ is an interior of the interval  $P_l\cap \{\lambda F + (1-\lambda)F_{A_1^l}\mid \lambda\in\R\}$), then, for any element $(\underline{s})\in \Sigma_{[l/2]}$,
the image of the correspondence  
\begin{equation}
\label{translation}
\begin{array}{c}
F \quad \mapsto \quad  F\ +\ \sum_{i=1}^{[l/2]} s_i F_{A_1}^{l-2i} F_\infty^i 
\end{array}
\end{equation} 
belongs to $P_l$ again. 
\end{lem}
\begin{proof}
The assumption on $F$ means that there exists $G\in P_l$ such that $F$ is an interior of the interval $[F_{A_1^l},G]$. That is, $F=\lambda F_{A_1^l}+(1-\lambda)G$ for some $\lambda\in (0,1)$. Then $P_l \ni   \lambda F_{I_l(\underline{s}/\lambda)} + (1-\lambda)G =F +\lambda( F_{I_l(\underline{s}/\lambda)}-F_{A_1^l})=F+F_{I_l(\underline{s})}-F_{A_1^l}$.
\end{proof}

Let us call \eqref{translation} the {\it translation action} of the element $(\underline{s})\in \Sigma_{[l/2]}$ on $F$. It is clear that $\Sigma_{[l/2]}$ still acts again on the element after an action, and that one has the associativity low for the composition of the actions. 

Obvously, $F_{A_1^l}$-interior of $P_l$ contains  interior of $P_l$, we obtain

\begin{corollary}
The  $F_{A_1^l}$-interior of $P_l$, and hence, the interior $\overset{\circ}{P}_l$ 
of $P_l$, is invariant under the translation action of the semi-group $\Sigma_{[l/2]}$.
\end{corollary}
\vspace{-0.1cm}
On the other hand, it is straight forward to see the following compactness.
\vspace{-0.1cm}

\begin{lem}
The quotient set $P_l/ \sum_{i=1}^{[l/2]} \R F_{A_1}^{l-2i} F_\infty^i $ is a compact convex polyhedron.
\vspace{-0.1cm}
\end{lem}
\begin{proof} We have to show that the quotient set is bounded and closed. But it is obvious, since, after the action of the semi-group $\Sigma_{[l/2]}$, there are only finite number of ``obits" of $\Z[Cox]_l$ for each $l$ (recall the proof of Proposition 5.), so that the quotient set is the convex-hull of their (finite) images. 
\end{proof}

In order to get precise description of the polyhedron $P_l$, we need precise data of linear or algebraic dependence relations among $F$-triangles, which seems rather intricate problem. 
In the next section, we show one approach to this problem of finding relations among $F$-triangles, which helps to embed $P_l$ into a lower dimensional affine space. 
\vspace{-0.2cm}


\section{ $A$-triangles.}

For any $\Phi \in \R_{\ge0}[Cox]_{l}$, we introduce an $A$-triangle $A_\Phi$, which, in some sense, is extracting  some core information of $F_\Phi$.  We hope that $A$-triangles should help finding linear dependence relations among $F$-triangles. We proceed this  program for rank 3 and 4 cases, but further rank cases seem still complicated.

\begin{lem}
\label{A-triangle}
For any polynomial $\Phi \in \R_{\ge0}[Cox]_{l}$ homogeneous of rank $l>0$, there exists a unique polynomial $A_\Phi(x,y)\in \R[x,y]$ such that
\begin{equation}
\vspace{-0.1cm}
\label{Atriangle}
F_\Phi(x,y) = A_\Phi(x,y) F_\infty +\Tr(\Phi) \ \B_l(x,y).
\end{equation}
\vspace{-0.1cm}
where $\B_l$ is a polynomial in $\R[x,y]$, independent of $\Phi$, given by
\begin{equation}
\vspace{-0.1cm}
\label{Btriangle}
\B_l(x,y) = ((-y)^l -(1-y)^l)x+ (1-y)^l .
\end{equation}
The polynomial $A_\Phi$, which we shall call the $A$-triangle part of the $F$-triangle $F_\Phi$,  satisfies the following properties.
\vspace{-0.1cm}
$$
\begin{array}{rllll}

\smallskip
{\rm i)} &  A_\Phi(x,y)=(-1)^lA_\Phi(1-x,1-y), \\

\medskip
{\rm ii)} & \partial_y A_{\Phi(I)}(x,y)=  - \sum_{\alpha\in I} A_{\Phi(I\setminus\{\alpha\})}(x,y), \\

\smallskip
{\rm iii) } & \text{By setting $A_\Phi(-x,-y)=\sum_{k,m} a_{k,m}x^ky^m$, all coefficients $a_{k,m}$} \\
&\!\!\! \text{are non-negative, $a_{k,m}\!>\!0 \text{ if } k\!+\!m \le l-2$ and $a_{k,m}\!=\!0$ if $k\!+\!m>l-2$. }
\end{array}
\vspace{-0.1cm}
$$
\end{lem}
\begin{proof}
Since $F_\infty$ is a monic polynomial of degree 2 in the variable $x$, we apply the Euclid division algorithm to $F_\Phi$ so that we obtain a unique expression:
\vspace{-0.1cm}
$$
F_\Phi(x,y) = A_\Phi(x,y) F_\infty + U(y)x +V(y)
\vspace{-0.1cm}
$$
for some polynomials $A_{\Phi}\in \R[x,y]$ and $U,V\in \R[y]$. Substituting $x=0$ and $x=1$, and applying \vspace{-0.1cm}
the formulae \eqref{x=01}, we obtain 
\vspace{-0.1cm}
$$
\vspace{-0.1cm}
\Tr(\Phi)(-y)^l=F_\Phi(1,y) = U(y) +V(y)  \text{\ \ and \ \ } 
\Tr(\Phi)(1-y)^l=F_\Phi(0,y) = V(y),
$$
implying the formulae \eqref{Atriangle} and \eqref{Btriangle}. 

Applying the rotational symmetry \eqref{rotation} to \eqref{Atriangle}, we obtain ii).
The formula \eqref{partialderivative} implies iii), where we use a fact $\partial_y\B_l=-l \B_{l-1}$. For a proof of iv), transform $(x,y)\mapsto (-x,-y)$ and write
$$
\begin{array}{c}
F_\Phi(-x,-y) = A_\Phi(-x,-y) x(x+1) +\Tr(\Phi) \ (\sum_{i=0}^{l-1}C_{l,i}y^i x+(1+y)^l), 
\end{array}
$$
where  LHS is the generating function of the simplicial cone counting of the cluster fan $\Delta(\Phi)$. The two terms of RHS has common degree only for the linear term in $x$. 
Note that there are $C_{l,i}$ number of $i$-dimensional cones  of the simplicial cone over $\Delta(\Phi)_{-}$, and that, for each cone, there exists at least one positive root such that the cone and the root together span a simplicial cone in the cluster fan $\Delta(\Phi)$ (see \cite{F-Z1} and \cite{C1}). Thus, the difference  LHS minus the second term of RHS is a polynomial of positive coefficients. Its quotient by $x(x+1)$ is still a positive coefficients polynomial (see [ibid]).
\end{proof}
Inspired by {\bf Lemma} \ref{A-triangle}, we introduce the space $A$-polynomials in $\R[x,y]$.

\begin{definition}
For $l\in \Z_{\ge0}$, we introduce the space $\A_l$ of $A$-polynomials as follows.
$
\A_0 :=\ \{1\}, \  \A_1 :=\{1-x-y\} 
$
and, for $l\in\Z_{\ge2}$, we set
$$
\A_l:= \ \{ A(x,y) F_\infty  + \B_l \mid A(x,y)\in \R[x,y] \text{ satisfies the following i) and iii)}  \}.
$$
In other words, a polynomial $F\in \R[x,y]$ of degree $l$ belongs to $\A_l$ and called a $A$-polynomial, if and only if  its residue part w.r.t.\ the Euclidean division (as a polynomial in $x$) by $F_\infty$ is equal to $\B_l$, and the quotient part $A(x,y)$ satisfies 


\ \ i) \ rotational symmetry: $A(1-x,1-y)=(-1)^lA(x,y)$, 

iii) \  non-negativity: $A(-x,-y)=\sum_{k,m}a_{k,m}x^ky^m$ s.t. $a_{k,m}\ge0\  \forall k,m$ ,

\quad\ \  $a_{k,m}>0 \text{ if } k+m\le l-2$ and $a_{k,m}=0 \text{ if } k+m>l-2$. 

\noindent
Let us call $A(x,y)$ the {\it $A$-triangle part} of the $A$-polynomial $F$
\end{definition}

The following is an immediate consequence of the definition.

\begin{proposition}
i)  The union $\A:= \cup_{l=0}^\infty \A_l$ is closed under product.

ii)  There is a semi-group action
\vspace{-0.2cm}
$$ 
\begin{array}{c}
\Sigma_{[l/2]}\times \A_l \to \A_l, (\underline{s}, A)\mapsto A+\sum_{i=1}^{[l/2]} s_iF_{A_1}^{l-2i}F_\infty^{i-1}.
\end{array}
\vspace{-0.2cm}
$$

iii) The normalized derivation 
\vspace{-0.2cm}
\begin{equation}
\label{D-derivation}
D\ :\   \R[x,y]\to \R[x,y], \quad F \ \mapsto \ -\frac{1}{l}\partial_y F
\vspace{-0.2cm}
\end{equation}
maps the set $\A_l$ to the set $\A_{l-1}$ ($l\in \Z_{>0})$.
\end{proposition}
\begin{proof}
i) We calculate directly 
\vspace{-0.1cm}
$$
\begin{array}{ccl}
\vspace{0.1cm}
&&(A_1(x,y) F_\infty  + \B_{l_1})(A_2(x,y) F_\infty  + \B_{l_2})\\ 
&=&\{A_1(x,y)A_2(x,y)F_\infty+A_1(x,y)\B_{l_2}+A_2(x,y)\B_{l_1} \\
&& +((1-y)^{l_1}-(-y)^{l_1})(  (1-y)^{l_2}-(-y)^{l_2})\}F_\infty +\B_{l_1+l_2}
\end{array}
$$
Each term in $\{\}$ satisfies i), ii) and iv). In particular, the  non-negativity of them implies the sum satisfies also i).

iii)  One sees directly  that  $D \B_l=\B_{l-1}$ and that if $A(x,y)$ satisfies the properties  i), ii) and iv) for $\A_l$ then $DA(x,y)$ preserves the properties i), ii) and iv) for $\A_{l-1}$.
\end{proof}

The space $\A_l$ of $A$-polynomials, where the polyhedron $P_l$ is embedded, seems to be a good frame to analyze $P_l$ and $F$-triangles (at least, for low rank cases).

In the rest of present section, we list up  $A$-triangle parts of $F$-triangles of rank $l\le 4$ and also of Chapoton family $F_{Chap(h)}$. The part of $A$-triangle is expressed by the boldface characters.

\bigskip
\noindent
{\bf Rank 2 $A$-triangles.} 

\noindent
Applying the relation $F_{I_2(2)}=(1-x-y)^2=F_\infty +\B_2$ to (3.5), we get
$$
\vspace{-0.1cm}
\begin{array} {rcl}
F_{I_2(p)}& =& F_{I_2(2)}+sF_\infty   \qquad\qquad \qquad\qquad\qquad \\
& =& {\bf  (s+1)} F_\infty + \B_2,\    \quad for \ \ s=p-2\ge 0.
\end{array}
\vspace{-0.1cm}
$$
This means that $P_2$ is a closed sub-half-line of $\A_1$ starting from 1 (but not 0).

\bigskip
\noindent
{\bf Rank 3 $A$-triangles.}
$$
\begin{array} {rcl}
F_{A_3}&=& 1-6x-3y+10x^2+8xy+3y^2 -5x^3-5x^2y-3xy^2-y^3\\
  &=& {\bf 5(1-x-y)}F_\infty + \B_3 \\
\vspace{0.2cm}
  & = & ({\bf 5X+5Y})F_\infty +\B_3.\\
F_{B_3}&=&1-9x-3y+18x^2+9xy+3y^2 -10x^3-6x^2y-3xy^2-y^3\\
  &=& {\bf 2(4-5x-3y)} F_\infty + \B_3  \\
\vspace{0.2cm}
 & = & ({\bf 10X+6Y})F_\infty +\B_3.\\
F_{H_3}&=&1-15x-3y+35x^2+10xy+3y^2 -21x^3-7x^2y-3xy^2-y^3\\
  &=& {\bf 7(2-3x-y)} F_\infty + \B_3 \\
\vspace{0.2cm}
 &= & ({\bf 21X+7Y})F_\infty +\B_3.\\
F_{A_1I_2(p)}&=&1 - (1+ p) x + (2 p-1) x^2  - (p-1) x^3 - 3 y + (4 + p) x y\\
\vspace{0.2cm}
&& -(p+1) x^2 y + 3 y^2 - 3 x y^2 - y^3\\
\vspace{0.2cm}
  &=&{\bf  (p-(p-1)x-(p+1)y)} F_\infty + \B_3  \\
\vspace{0.2cm}
  &= & ({\bf (p-1)X+(p+1)Y})F_\infty +\B_3.\\ 
  F_{Chap(h)} &=& 1 - \frac{3 h}{2} x + \frac{3 h (-2 + 3 h)}{2(2+h)} x^2 - \frac{(-1 + h) (-2 + 3 h)}{2+h} x^3 - 3 y + \frac{12 h}{2+h} x y \\
\vspace{0.2cm}
  && - \frac{3 (-2 + 3 h)}{2+h} x^2 + 3 y^2 - 3 x y^2 - y^3\\
\vspace{0.2cm}
  &=& ( {\bf \frac{3h-2}{2} - \frac{(h-1)(3h-2)}{h+2}x -\frac{3(3h-2)}{h+2}y}) F_\infty +\B_3 \\
  &=&  {\bf (\frac{(h-1)(3h-2)}{h+2}X +\frac{3(3h-2)}{h+2}Y)} F_\infty +\B_3
\end{array}
$$

Here, $X:=1/2-x$ and $Y:=1/2-y$ are the basis of space of polynomials in $x,y$ of degree 1 with the rotational symmetry (Lemma 4.1 ii)), so that 
\vspace{-0.1cm}
$$
\A_3\ \ = \ \ \R_{>0}XF_\infty+ \R_{>0}YF_\infty +\B_3.
\vspace{-0.1cm}
$$
We draw, in the following Figure 1, the locations of $P_3$ and the trace of Chapoton family $F_{Chap(h)}$  in the $\A_3$-plane. 
For the explanations of the two more curves and the domain names I, II and III, see \S5, Rank 3 case (see Figure 5).


\begin{figure}[h]
\center
\includegraphics[width=11.0cm]{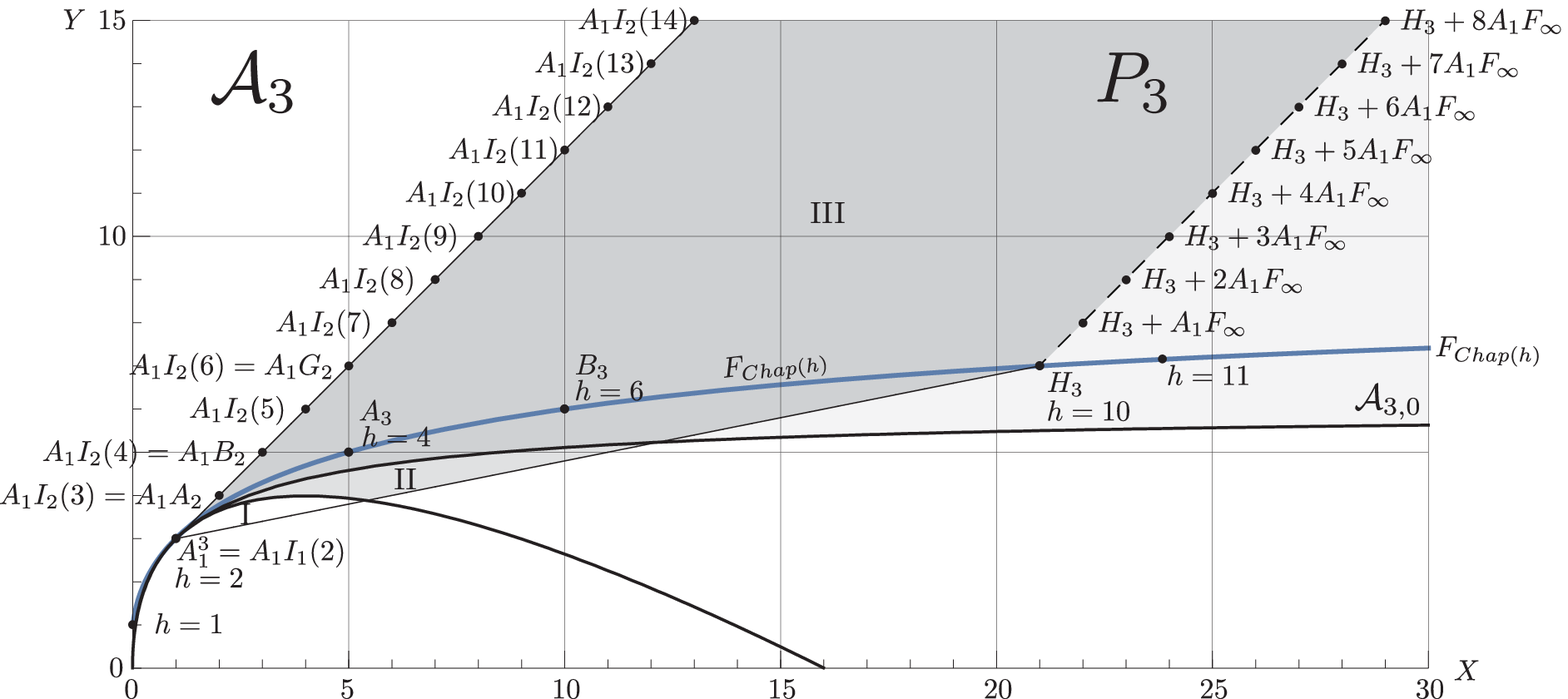}
\vspace{-0.3cm}
\caption{\ Polyhedron $P_3$ and the Chapoton family $F_{Chap(h)}$ in $\A_3$}
\vspace{-0.4cm}
\end{figure}

\vspace{-0.2cm}

\noindent
{\bf Rank 4 $A$-triangles.}
$$
\begin{array} {rcl}
\vspace{-0.01cm}
F_{A_4}&=& 1-10 x+ 30 x^2 - 35 x^3 + 14 x^4 - 4 y + 20 x y - 30 x^2 y + 14 x^3 y \\
&& + 6 y^2 - 15 x y^2 + 9  x^2 y^2  -  4 y^3 +  4 x y^3  + y^4 \\
\vspace{0.11cm}
  &=& {\bf (9 - 21 x + 14 x^2 - 16 y + 14 x y + 9 y^2)} F_\infty + \B_4.\\

\vspace{-0.01cm}
F_{B_4}&=&1 - 16 x + 60 x^2 - 80 x^3 + 35 x^4 - 4 y + 24 x y - 40 x^2 y + 
 20 x^3 y \\
 &&+ 6 y^2 - 16 x y^2 + 10 x^2 y^2 - 4 y^3 + 4 x y^3 + y^4 \\
\vspace{0.11cm}
  &=&{\bf ( 15- 45 x + 35 x^2 - 20 y + 20 x y + 10 y^2)} F_\infty + \B_4.\\

\vspace{-0.01cm}
F_{D_4}&=&1 - 12 x + 39 x^2 - 48 x^3 + 20 x^4 - 4 y + 21 x y - 33 x^2 y + 
 16 x^3 y \\
 &&+ 6 y^2 - 15 x y^2 + 9 x^2 y^2 - 4 y^3 + 4 x y^3 + y^4 \\
\vspace{0.11cm}
  &=&{\bf (11 - 28 x + 20 x^2 - 17 y + 16 x y + 9 y^2)} F_\infty + \B_4.\\

\vspace{-0.01cm}
F_{F_4}&=&1 - 24 x + 101 x^2 - 144 x^3 + 66 x^4 - 4 y + 26 x y - 46 x^2 y + 
 24 x^3 y \\
 &&+ 6 y^2 - 16 x y^2 + 10 x^2 y^2 - 4 y^3 + 4 x y^3 + y^4 \\
\vspace{0.11cm}
  &=& {\bf (23 - 78 x + 66 x^2 - 22 y + 24 x y + 10 y^2)} F_\infty + \B_4.\\
  
\vspace{-0.01cm}
  F_{H_4}&=&1 - 60 x + 307 x^2 - 480 x^3 + 232 x^4 - 4 y + 31 x y - 59 x^2 y + 
 32 x^3 y\\
 && + 6 y^2 - 17 x y^2 + 11 x^2 y^2 - 4 y^3 + 4 x y^3 + y^4\\
\vspace{0.11cm}
  &=& {\bf (59 - 248 x + 232 x^2 - 27 y + 32 x y + 11 y^2)} F_\infty + \B_4.\\

\vspace{-0.01cm}
F_{A_1A_3}&=&1 - 7 x + 16 x^2 - 15 x^3 + 5 x^4 - 4 y + 17 x y - 23 x^2 y + 
 10 x^3 y \\
 && + 6 y^2 - 14 x y^2 + 8 x^2 y^2 - 4 y^3 + 4 x y^3 + y^4\\
 \vspace{0.11cm}
  &=& {\bf (6 - 10 x + 5 x^2 - 13 y + 10 x y + 8 y^2)} F_\infty +\B_4 .\\

\vspace{-0.01cm}
F_{A_1B_3}&=&1 - 10 x + 27 x^2 - 28 x^3 + 10 x^4 - 4 y + 21 x y - 33 x^2 y + 
 16 x^3 y \\
 && + 6 y^2 - 15 x y^2 + 9 x^2 y^2 - 4 y^3 + 4 x y^3 + y^4\\
\vspace{0.11cm}
  &=&{\bf (9 - 18 x + 10 x^2 - 17 y + 16 x y + 9 y^2)} F_\infty + \B_4.\\

\vspace{-0.01cm}
F_{A_1H_3}&=&1 - 16 x + 50 x^2 - 56 x^3 + 21 x^4 - 4 y + 28 x y - 52 x^2 y + 
 28 x^3 y\\
 && + 6 y^2 - 16 x y^2 + 10 x^2 y^2 - 4 y^3 + 4 x y^3 + y^4\\
\vspace{0.11cm}
  &=&{\bf (9 - 18 x + 10 x^2 - 17 y + 16 x y + 9 y^2)} F_\infty + \B_4.\\

F_{I_2(p)I_2(q)}&=&(1 - p x + (-1 + p) x^2 - 2 y + 2 x y + y^2)\\
&& \times  (1 -  q x + (-1 + q) x^2 - 2 y + 2 x y + y^2)\\
  &=&{\bf \big(-1 + p + q + (p - 1)( q-1) x^2 - 2 (p+q) y + (2 + p + 
 q) y^2 }\\
\vspace{0.11cm}
 &&{\bf  -(p q-1)x  + 2( p+ q-2) x y)\big)} F_\infty + \B_4.\\

\vspace{-0.01cm}
F_{I_4(s_1,s_2)}&=&(1 -  x -y)^4+s_1(1- x -  y)^2x (x-1)+s_2x^2(x-1)^2\\
  &=&{\bf \big(3 + s_1 -(3 +2s_1 + s_2) x + (1+s_1+ s_2) x^2 - (8+2s_1) y }\\
 &&{\bf + (6+s_1) y^2  + (4+2s_1) x y  \big)} F_\infty + \B_4.\\
 \end{array}
$$
The space of polynomials in $x,y$ of degree $\le2$ with rotational symmetry is a vector space of rank 4 spanned by $x(x-1),\ 2xy-x-y,\ y(y-1)$ and 1. So the  $P_4$ is an un-bounded semi-algebraic  polygon in the 4  dimensional space $\A_4$:
$$
\R_{\ge0}x(x-1)F_{\infty}+\R_{\ge0}(2xy-x-y)F_{\infty}+\R_{\ge0}y(y-1)F_{\infty}+\R_{\ge0}F_{\infty}+\B_4.
$$ 
Let $(a,b,c,d)\in \R^4$ be coordinates of the point $ax(x-1)F_{\infty}+b(2xy-x-y)F_{\infty}+cy(y-1)F_{\infty}+dF_{\infty}+\B_4$ of the affine space. Then, the semi-group element $(s_1,s_2)\in \Sigma_2$ acts on the coordinates by the translation (recall \eqref{translation}):
\begin{equation}
\label{action4}
(a,b,c,d) \mapsto (a+s_1+s_2,b+s_1,c+s_1,d+s_1) .
\end{equation}
An explicit coordinate values for all finite Coxeter systems of rank  4 are given in the left side of the following table (recall Lemma 4.1 iv) for the non-negativity).
$$\begin{array}{rrrrrc}
\vspace{-0.1cm}
 \mid \!\!\!\!\!  & a & b & c & d  &  (\lambda,\mu,\nu) + s_1F_\infty^2 + s_2 F_{A_1}^2F_\infty \\
 -----\!  \mid \!\!\!\!\! &\!\!\! -----\! \!\!\!\!&\!\!\!\!--- \!\!\!\!&\!\!\!\!\! --- \!\!\!&\!\!\!\! ---\!\!\!\!&\!\! -------------\\
A_4  \ \ \mid \!\!\!\!\!  &  14 & 7 & 9 & 9 &\!\!    (\frac{253}{336},\frac{25}{112},\frac{1}{42})+ \frac{167}{84}F_\infty^2 + \frac{22}{21} F_{A_1}^2 \!F_\infty\\
B_4  \ \ \mid \!\!\!\!\!  &  35 & 10 & 10 & 15 &    (\frac{43}{84},\frac{11}{28},\frac{2}{21})+ \frac{41}{21}F_\infty^2 + \frac{46}{21} F_{A_1}^2F_\infty\\
D_4  \ \ \mid \!\!\!\!\!  &  20 & 8 & 9 & 11 &    (\frac{53}{84},\frac{9}{28},\frac{1}{21})+ \frac{31}{21}F_\infty^2 + \frac{2}{21} F_{A_1}^2F_\infty\\
F_4  \ \ \mid \!\!\!\!\!  & 66  & 12 & 10 & 23 &    (\frac{47}{168},\frac{27}{56},\frac{5}{21})+ \frac{37}{42}F_\infty^2 -\frac{11}{21} F_{A_1}^2F_\infty\\
H_4  \ \ \mid \!\!\!\!\!  & 232  & 16 & 11 & 59 &    (0,0,1)+ 0F_\infty^2 + 0 F_{A_1}^2F_\infty\\
A_1A_3  \ \ \mid \!\!\!\!\!  & 5  & 5 & 8 & 6 &    (\frac{7}{8},\frac{1}{8},0)+ \frac{3}{2}F_\infty^2 + 0 F_{A_1}^2F_\infty\\
A_1B_3  \ \ \mid \!\!\!\!\!  & 10  & 9 & 8 & 9 &    (\frac{5}{8},\frac{3}{8},0)+ \frac{3}{2}F_\infty^2 + 0 F_{A_1}^2F_\infty\\
A_1H_3  \ \ \mid \!\!\!\!\!  & 21  & 14 & 10 & 15 &    (0,1,0)+ 0F_\infty^2 + 0 F_{A_1}^2F_\infty\\
I_2(p)I_2(q)    \mid \!\!\!\!\!  &\! \!\!\!\! (p\!\!-\!\!\!1\!)(q\!\!-\!\!\!1\!) \!\! &\!\!\! p\!\!+\!\!q\!\!-\!\!\!2 \!\!\! & \!\!\! p\!\!+\!\!q\!\!+\!\!\!2\!\! \! &\! \! p\!\!+\!\!q\!\!-\!\!\!1 \!\! \! & \!  (1,\!0,\!0) \!\!+ \!\! (p\!\!+ \!\!q \!\! - \!\! 4)F_\infty^2 \!\! + \!\! (pq\!\! - \!\!3) F_{A_1}^2F_\infty  \\
I_4(s_1,s_2)  \ \ \mid \!\!\!\!\!  & 1\!+\!s_1\!+\!s_2 \!\! & 2\!+\!s_1 \!\!\!\! & 6\!+\!s_1 \!\!\!\! & 3\!+\!s_1 \!\!\!\! &  (1,0,0) + s_1F_\infty^2 + s_2 F_{A_1}^2F_\infty\\
\end{array}
$$
Combining data of this Table with the description of the semi-group action \eqref{action4}, one observes that the quotient $P_4/ (\R F_{A_1^2} F_\infty+\R F_\infty^2)$ is a triangle spanned by the images of $I_4(0,0)=A_1^4, A_1H_3$ and $H_4$ (see Figure 2). 

\begin{figure}[h]
\center
\setlength\unitlength{3.5cm}
\begin{picture}(2.0,0.95)(-.5,-.02)

\put(0,0){\line(1,0){1}}
\put(0,0){\line(3,5){0.5}}
\put(1,0){\line(-3,5){0.5}}

\put(0,0){\circle*{.02}} \put(-.290,-.010){$A_1H_3$}
\put(.5,.833){\circle*{.02}} \put(-.09,.823){$A_1^4=I_4(0,0)$}
\put(1,0){\circle*{.02}} \put(1.050,-.010){$H_4$}
\put(.400,.627){\circle*{.02}} \put(.400,.617){ $A_4$}

\put(.351,.426){\circle*{.02}} \put(.358,.41){ $B_4$}
\put(.363,.525){\circle*{.02}} \put(.37,.515){ $D_4$}
\put(.377,.233){\circle*{.02}} \put(.38,.22){ $F_4$}
\put(.437,.729){\circle*{.02}} \put(.16,.71){$A_1A_3$}
\put(.312,.520){\circle*{.02}} \put(.030,.508){$A_1B_3$}

\end{picture}
\caption{ Polyhedron  $P_4/ (\R F_{A_1^2} F_\infty+\R F_\infty^2)$ in $\A_4/\Sigma_2$}
\end{figure}
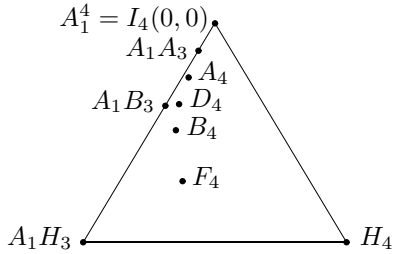

Then using the barycentric coordinates $(\lambda,\mu,\nu)$ with respect to the vertices $I_4(0,0), A_1H_3$ and $H_4$, all A-triangles of rank 4 are described in RHS of Table. In the description, only $F_4$ obtained negative coefficient $s_2<0$. this means that all other A-triangles are below the triangle spanned by $I_4(0,0), A_1H_3$ and $H_4$ (that is, they  are in the image of the semi-group $\Sigma_2$-action on the triangle). 

\medskip
{\noindent}
{\bf Fact .} {\it The $P_4$ is a 4-dimensional polyhedron obtained as the union of images of the translations of the 3-simplex $[A_{F_4},A_{I_4(0)}, A_{A_1H_3},A_{H_4}]$ by the action of $\forall (s_1,s_2)\in \Sigma_2$, whose three dimensional closed face obtained  as the union of the  images of translation of the 1-simplex  $[A_{A_1H_3},A_{H_4}]$ by the action of $\forall (s_1,s_2)\in \Sigma_2$ should be removed.}

\begin{rem}
In the polyhedrons $P_3$ and $P_4$, the points corresponding to finite irreducible root systems appear interior of the polyhedron. On the other hand, non-crystallographic Coxeter groups,  as the extremal points, seem to play role to determine the polygon.  This seems slightly disappointing in the sense that the subject does not depend on the fine combinatorics of finite root systems but on something else. Nevertheless, the original expectation stated in {\bf Abstract} remains meaning full, and we start to analyze examples from next section.
\end{rem}


   

\section{Zero loci of $F$-triangles of rank $\le4$}

In connection with the expectations stated in {\bf Abstract}, we draw figures of the zero loci of $F$-triangles on the unit square $[0,1]\times[0,1]$ in the $x$-$y$ plane, and confirm pictorially the expectation for the cases of rank $\le4$, up to a modification at bending points introduced in the present section. Actually, it is elmentally to give mathematical proofs for those observations, which is left to the reader.

\bigskip
\noindent
{\bf Explanation of Figures 3, 4, 6, 7, 8, 9 and 10.}  

\noindent
The Figures exhibit zero-loci of a function $F \in \A_l$ ($l\! \le\! 4$) in the following way.

\smallskip
i)   The big square (the frame of the figure) 
exhibits a part of $x$-$y$ real plane, 

\quad where coordinate values are given on the boundary of the frame.

ii)  The small square 
exhibits the unit square $[0,1]\times[0,1]$  in the $x$-$y$ plane.

iii) The solid curve  
exhibits the zero loci of $F$ inside the frame.

iv) The dashed curve  
exhibits the zero loci of $D F_\Phi$  {\small(see $*$))} inside the frame 

\quad (see Conjecture 5. and Discussions 2. and 9. in \S6).

v) The dotted curve  
exhibits the zero loci of $D^2F_\Phi$ {\small(see $*$))} inside the frame 

\quad (see Conjecture 6. and Discussion 2. in \S6).

vi) The type (or, a label) of the function $F$  is given at the upper right corner.

\noindent
$*$)  \ \ Here, we recall the normalized derivation $D:\A_l\to \A_{l-1}$ \eqref{D-derivation}.

\medskip
\noindent
{\bf 1.  Rank 1 case (Figure 3).}

Recall the descriptions \eqref{p1} of $P_1=\{F_{A_1}\!=\!1\!-\!x\!-\!y\}$. The zero-loci of $F_{A_1}$ induces the anti-diagonal line on $[0,1]\times[0,1]$ (see Figure 3).
 
\medskip
\noindent
{\bf 2.  Rank 2 case (Figure 3).}

\vspace{-0.05cm}
Recall the description  \eqref{p2} of $P_2=\{F_{I_2(p)}=F_{A_1^2}+(p-2)F_\infty\mid p\in\R_{\ge2}\}$. The  zero-loci of $F_{I_2(p)}$ for $p>2$  is an ellipse tangent to lines $x\!=\!0$ and $x\!=\!1$ at $y\!=\!1$ and $y\!=\!0$, respectively, which intersects with the intervals $[0,1]\! \times\! 1$ and $[0,1]\! \times\! 0$ at $x\!=\!\frac{p\!-\!2}{p\!-\!1}$ and $x\!=\!\frac{1}{p\!-\!1}$, respectively. They satisfy the expectations. 
The zero-loci of $F_{I_2(2)}$ is a double anti-diagonal lines, which, up to the simplicity, satisfies the expectations. 
If $p<2$, then the zero locus of $F_{I_2(p)}$ is a hyperbola (see the dashed curve in Figure 8, $h=1$), and does not satisfy the expectations.

\begin{figure}[h]
\center
\includegraphics[width=12.3cm]{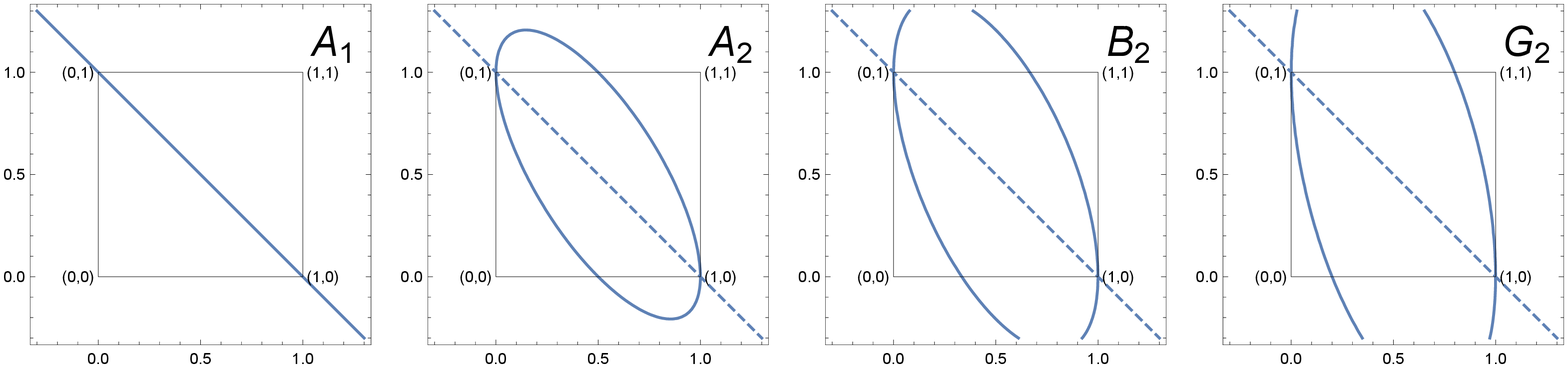}
\caption{  $F$-triangles of rank 1 and 2 of types $A_1$, $A_2$, $B_2$ and $G_2$ }
\vspace{-0.2cm}
\end{figure}


\bigskip
\noindent
{\bf 3.  Rank  $3$ case.} 

\noindent
{\bf  1) Figure 4:} We first study the cases of $F$-triangles $F_\Phi$ for a finite Coxeter group of  type $\Phi$ of rank 3. Among the infinite sequence of types $A_1I_2(p)$ ($p\ge2$), we exhibit only  types $A_1A_2$, $A_1B_2$ and $A_1G_2$ since other cases behave similarly.

\begin{figure}[h]
\center
\includegraphics[width=11.cm]{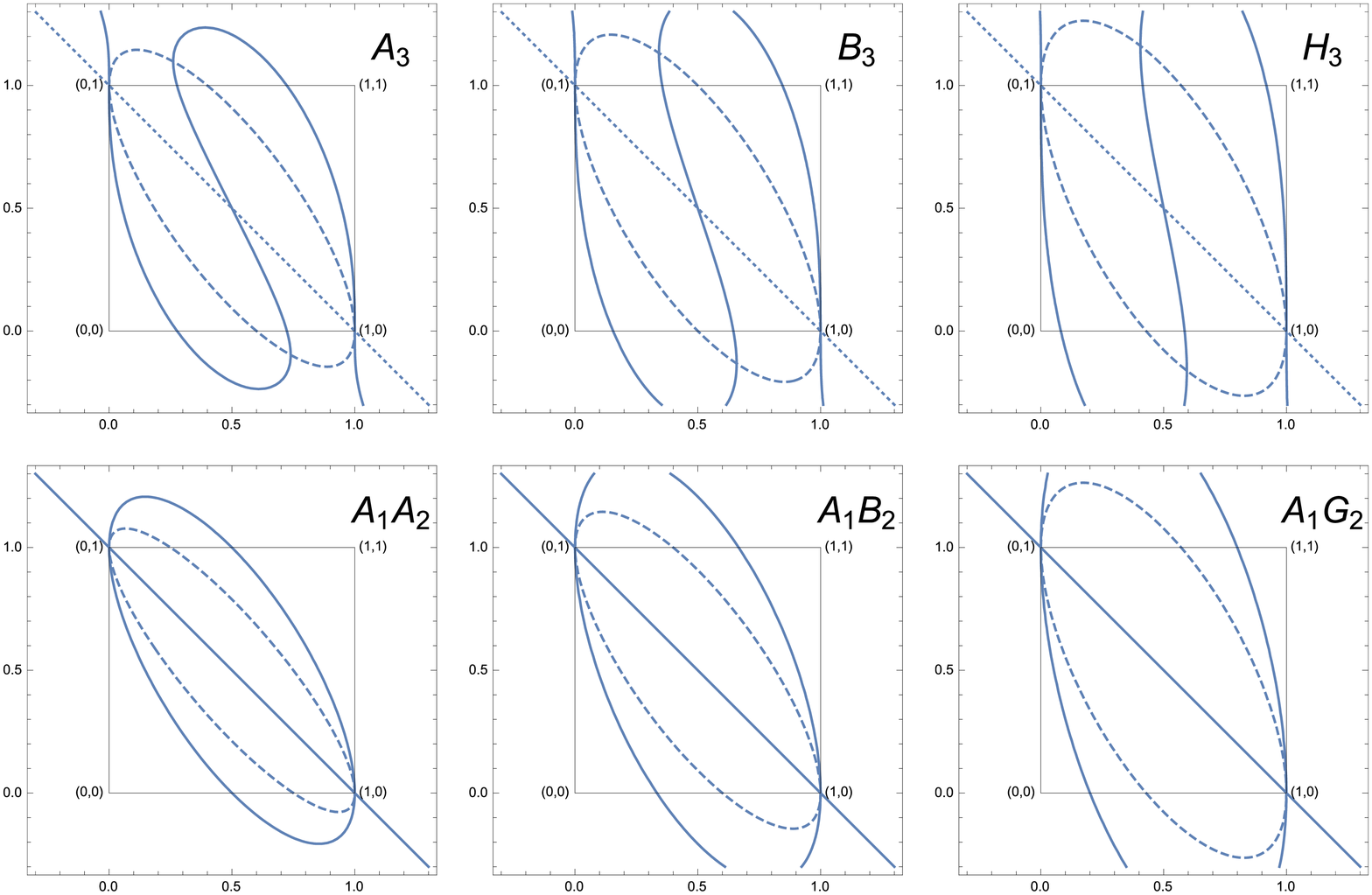}
\caption{$F$-triangles of rank 3 ($A_3$, $B_3$, $H_3$, $A_1A_2$, $A_1B_2$ and $A_1G_2$)}
\vspace{-0.2cm}
\end{figure}
\noindent
In order to proceed precise discussions, we formulate the expectations explicitly:

a)  For all $y\in [0,1]$, all roots of $F_\Phi|y$ lie in $[0,1]$ and  are simple.

b) The $i$-th root $x_i(y)$ is decreasing in the strong sense $\frac{\partial x_i}{\partial y}\!<\!0$  on $(0,1)$.

\noindent
$\Leftrightarrow$ b)'  The roots $x_i'(y)$  of $DF_\Phi|_y$ separates (strictly) the roots of $F_\Phi|y$.
\vspace{-0.2cm}
$$\begin{array}{l}
0\le x_1(y)<x_1'(y)<x_2(y)<x_2'(y)\cdots x_{l-1}'<x_l(y)\le1, \\
\text{\small \!\!\!\!\!\!\!\!\!\!\!\!\!\!\!\!  except for $x_1(1)\!=\!x'_1(1)\!=\!0$, $x_l(0)\!=\!x'_l(0)\!=\!1$}, 
\end{array}
$$

\vspace{-0.15cm}
\noindent
$\Leftrightarrow$ b)"  
Graphs of $x_1,x_2$ and $x_3$ belong to different connected component of 
 \ \ $[0,1]\times[0,1]\setminus \{DF_\Phi=0\}$ up to the points $(0,1)$ and $(1,0)$.

\smallskip
Then, we observes that 1) irreducible types $A_3$, $B_3$ and $H_3$ satisfies both a) and b), and 2)  reducible types $A_1A_2$, $A_1B_2$ and $A_1G_2$ satisfies both a) and b) up to the multiple roots $x_1(1)=x_2(1)=0$ and $x_{l-1}(0)=x_{l}(0)=1$.

\medskip
\noindent
{\bf 2) Figure 5 and 6:} We show examples of $F\in P_3$, where either a) or b) above is not satisfied. 
Recall the description of $P_3$ in \S4:  any $F$-triangle in $P_3$ of rank 3 is expressed as  
\vspace{-0.3cm}
$$
F_{\lambda,s}:=(1-\lambda) F_{A_1^3}  + \lambda F_{H_3} + sF_{A_1}F_\infty
$$
for the parameters $\lambda \!\in\! [0,1]$ and $s\!\in \!\R_{\ge0}$ (where $\lambda\!=\!1$, $s\!>\!0$ is the border). It is a cubic polynomial in the variable $x$ with coefficients in $\R[y,\lambda,s]$, where the leading coefficient $-(1+20\lambda+s)$ is a  strictly negative.  It is elementary to see

a) The polynomials $F_{\!\lambda,s}\!|_y$ has $l$ roots$^{*)}$ in $[0,1]$ for all $y\!\in\![0,1]$ if and only if 
\vspace{-0.1cm}
$$
\begin{array}{l}
-32 \lambda + 144 \lambda^2 + 24 \lambda s + s^2\ge 0 \quad  \text{(shaded part $II\cup III$ in Figure 1 and 5)}. \\
\text{\small \!\!\!\!\! $*$) Roots are simple if the strict innequlity $>$ holds.}
\end{array}
\vspace{-0.1cm}
$$

b) The functions $x_i$ are decreasing in the strong sense$^{**)}$ if and only if 
\vspace{-0.1cm}
$$
\begin{array}{l}
-36 \lambda + 64 \lambda^2 + 20 \lambda s +  s^2\ge 0 \quad  \text{(dark shaded part $III$ in Figure 1 and 5)}. \\ 
 \text{\small \!\!\!\!\!\!\! $**$) This is up to at $y=0,1$, and holds everywhere if the strict innequlity $>$ holds.}
\end{array}
$$
Therefore, we decompose $P_3=I\cup II\cup III$ where $I:=\{F_{\lambda,s} | \text{ a) does not hold.}\}$, $II \!:=\! \{F_{\lambda,s}| \text{ a) holds but not b)}\}$, and $III \!:=\! \{F_{\lambda,s}| \text{ a) and b) hold}\}$ (Figure 5).   

Choose three points  $(\lambda=0.05,s=0.5) \in I, \ (\lambda=0.05,s=0.6649...)\in II$ and $(\lambda=0.05, s=0.8747...) \in III$. 

\begin{figure}[h]
\center
\includegraphics[width=4.0cm]{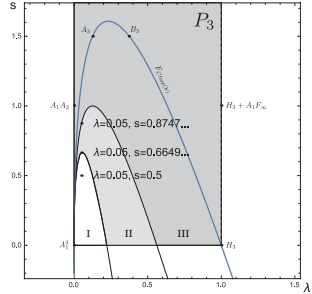}
\vspace{-0.3cm}
\caption{\ Polyhedron $P_3=I\cup II\cup III$ and $F_{Chap(h)}$}
\vspace{-0.3cm}
\end{figure}
\begin{figure}[h]
\center
\includegraphics[width=11.3cm]{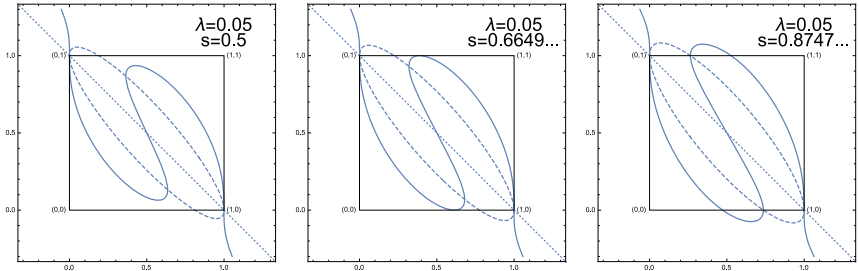}
\caption{  $F$-triangles in area I, II and III}
\end{figure}

{\bf 3) Figure 7:} We take three samples form the border $\overline{P_3}\setminus P_3$ and confirm that the expectation a) and b) are satisfied.
Recall that the function on the border of $P_3$ is described as $F=F_{H_3}+s F_{A_1}F_{\infty}$ for $s\in \R_{>0}$.  We choose three cases: $F_{H_3}+ F_{A_1}F_{\infty}$, $F_{H_3}+ 10 F_{A_1}F_{\infty}$ and $F_{H_3}+ 20 F_{A_1}F_{\infty}$ as for the test. The results are exhibited in Figure 7. 

\begin{figure}[h]
\center
\includegraphics[width=11.3cm]{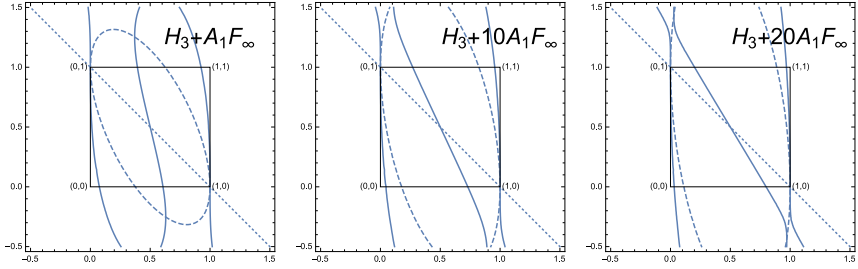}
\caption{Functions of type $H_3+sA_1F_\infty$ on the border of $\overline{P}_3$ }
\end{figure}

{\bf 4) Figure 8.}  We study the zero loci of  functions in the Chapoton family $F_{Chap(h)}$ (recall \S4 Table for rank 3 $A$-triangles).  We see immediately

\smallskip
  $\{h\in\R\mid F_{Chap(h)}\in \A_3\}=[1,\infty)$ \ \ and \ \ $\{h\in \R\mid F_{Chap(h)}\in P_3\}=[2,10]$ 
  
  \smallskip
  \noindent
 (cf.\ Figure 1 and 5). 
Therefre, we decompose the parameter space $[1,\infty)$ of the family into three pieces: $[1,\infty)=[1,2)\cup [2,10]\cup (10,\infty)$.
%
%
%
 
 1. The first component is ``out of range" in the sense that $DF_{Chap(h)}\not \in P_2$ for all $h\in [1,2)$. We choose one point, say $h=1$, as for a sample.
 
 2. We don't choose any point from the second component since it is in the area III in $P_3$ (see Fig. 1 and 5) where we know already a) and b) are satisfied.  
 
 3.  We choose two test points $h=11$ and $h=22$ from the third component. 
 
 The  resulting Figure 8 shows that the case $h=1$ satisfies non of a) and b), and the cases $h=11$ and $22$ satisfy both a) and b).

\vspace{-0.3cm}
\begin{figure}[h]
\center
\includegraphics[width=11.3cm]{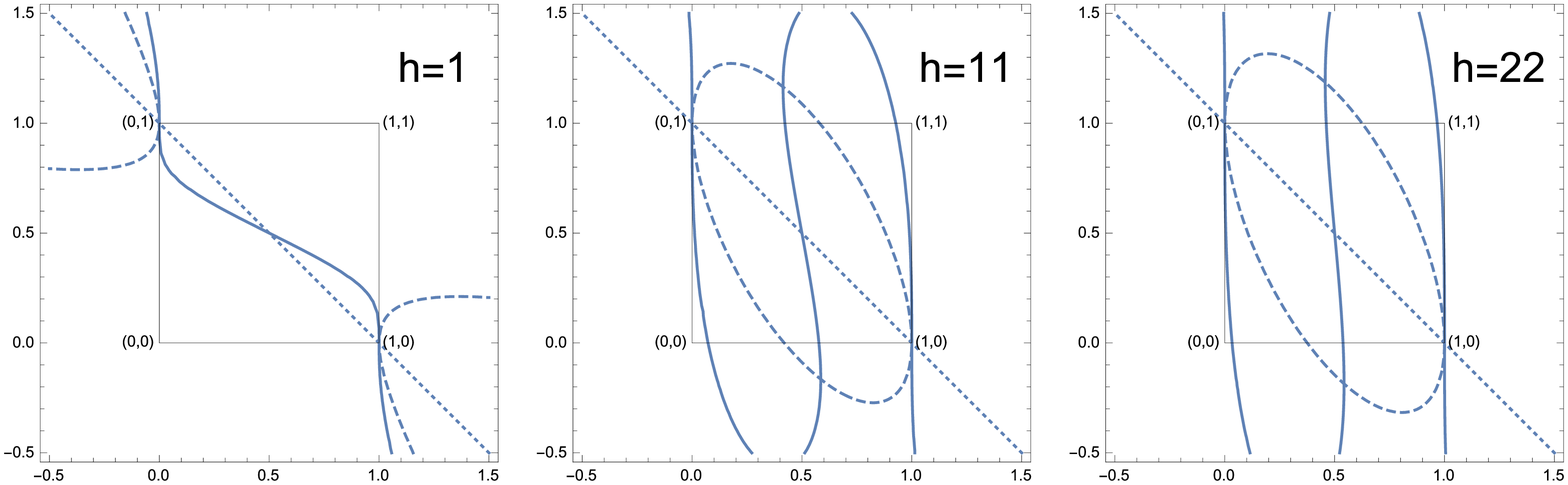}
\caption{Chapoton family $F_{Chap(h)}$}
\end{figure}
\vspace{-0.2cm}

\noindent
{\bf 4.  Rank  $4$ case.}

As in case of rank 3, we study the cases of $F$-triangles $F_\Phi$ for a finite Coxeter group of type $\Phi$ of rank 4.  There is an infinite family of types $I_2(p)I_2(q)$ ($p,q\ge2$). But since they are unions of two ellipses whose nature is well understood, we study only a typical case for $p=3$ and $q=6$. 

For a covenience of explanations, we divide the $F$-triangles into two groups:
$I:=\{A_4,B_4,D_4,F_4,H_4, I_2(p)I_2(q) \ \ (p,q\in \Z_{\ge2}\}$ and 
$II:=\{ A_1A_3, A_1B_3, A_1H_3\}$.

\begin{figure}[h]
\center
\includegraphics[width=12.0cm]{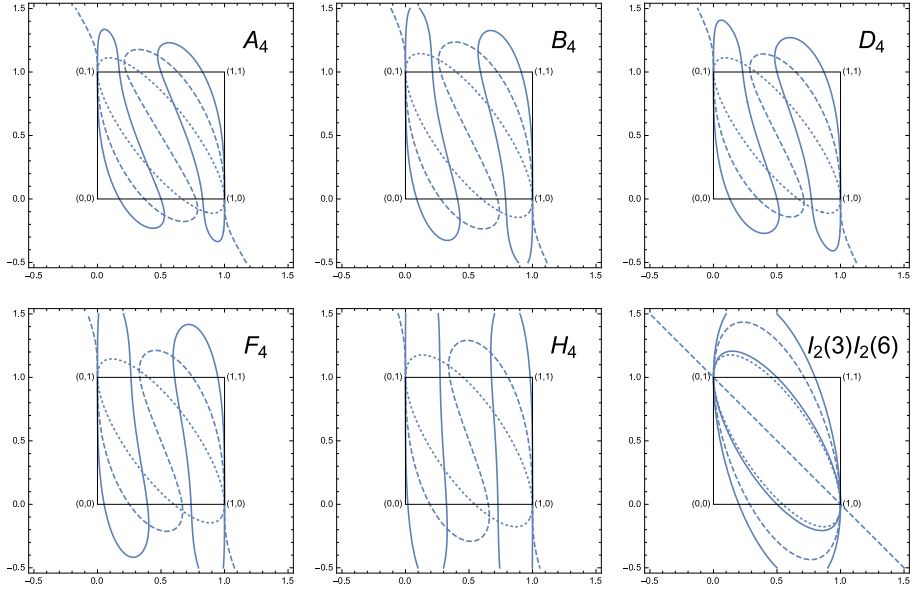}
\caption{ $F$-triangles of rank 4. I}
\end{figure}

\begin{figure}[h]
\center
\includegraphics[width=12.0cm]{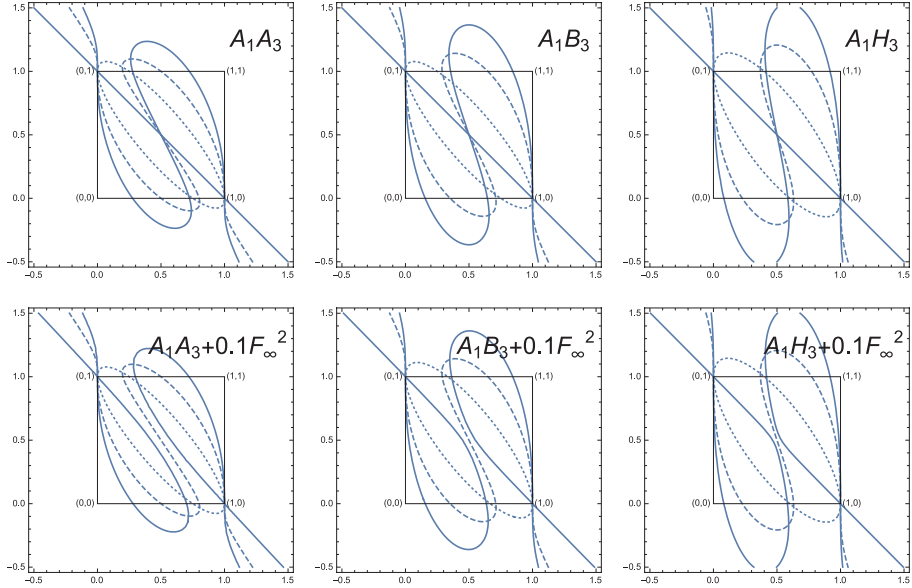}
\caption{ $F$-triangles of rank 4. II}
\end{figure}
We observe the following

\smallskip
1.  The $F$-triangles of type $A_4,B_4,D_4,F_4$ and $H_4$ satisfy the expectations a) and b). Their derivatives $DF$ (which are $F$-triangles in the generalized sense) satisfy also the expectations a) and b) (since the derivatives belong to the area III in $P_3$), and, then the second derivatives $D^2F$  again satisfy the expectation a) and b), and so on  (see Figure 9, where we omitted the zeros of $D^3F$).  

Actually, even though it was not mentioned explicitly, such ``inductive" structure can be already observed for the rank 3 cases (including types $A_3,B_3,H_3,$  $H+3+sA_1F_\infty\ (s\ge0), F_{Chap(h)}\ (h>2)$ and open doamin of III).

2.  The above mentioned ``inductive structure" can be neatly formulated as:

\smallskip
c)  The sequence $D^kF_\Phi|_y$ ($0\le k\le l$) form Sturm sequence for $y\in [0,1]$, 

\smallskip
\noindent
where the property b) in each inductive steps is a part of the defining condition of a Sturm sequence (see \S6 Conjecture A6, and \cite{T} \S16). 

3.  Above observations are valid for $F$-triangles of type $I_2(p)I_2(q)$ for $2\le p<q$ up to roots at $y=0$ and $1$. That is,  c) holds for $y\in (0,1)$ (see Figure 9).

4. For $F$-triangles of non-irreducible types $A_1A_3$, $A_1B_3$ and $A_1H_3$ (see upper row  of Figure 10),  the expectation a) is true only for $y\in [0,1]\setminus\{1/2\}$, but $F|_{1/2}$ has multiple roots $x_2(1/2)=x_3(1/2)=1/2$. Then,  b) and b)' are true only for $y\in [0,1]\setminus\{1/2\}$), since the functions of $x_2(y)$ and $x_3(y)$ are "bending" at $y=1/2$ so that the derivatives at the points are not defined, and we have the equality $x_2(1/2)=x_2'(1/2)=x_3(1/2)$.  Thus, b)" is true in the weaker sense that the graphs of $x_2$ and $x_3$ are note fully contained in the connected  components of $[0,1]\times [0,1]\setminus \{DF=0\}$, but are touching to the boundary of the components at the bending point $y=1/2$.  In summery, c) holds for $y\in [0,1]\setminus \{1/2\}$.

 
 5. On the other hand, it is interesting to observe that the semi-group $\Sigma_4$-action on these function of non-irreducible types $A_1A_3$, $A_1B_3$ and $A_1H_3$ deform them  to functions, for which  c) holds for $y\in [0,1]$ (see lower  row of Figure 10).

\medskip
{\bf A  summery of observations in \S5} 

1.  All $F$-triangles $F$ (up to the non-reduced cases $F_{A_1^l}$, $F_{A_1^2I_2(p)}$ and $F_{I_2(p)^2}$)  of rank $l \le4$  satisfy the expectation stated in Abstract, in a stronger inductive form: 
The seuence $D^kF|_y$ ($0\le k\le l$) form a Sturm sequence for $y$ in a dense subset of $[0,1]$, 
 
 \smallskip
2.  In the above 1,  the functions $D^kF$ belongs to the polyhedron $P_{l-k}$, but is not a $F$-triangle in the classical sense. That is, we need to extend the class of functions in order to get a ``self-closed'' inductively formulation  of the results.

\smallskip
3. Actually, functions belonging to some big part III of  $P_3$ satisfy this  propoerty. On the other hand there is some area I and  II of $P_3$, where the expectation does not hold, that is, $F|_y$ may not be totally real for some $y$ in anopensubset of $[0,1]$, or even  $F|_y$ is totally real for all $y\in [0,1]$, the function $x_i$ may not be monoton decreasing. 

\smallskip
4.  Above discussed  problems on $P_l$ seem naturally extend  to the problems on $\A_l$ (because of the boundary condition $F \bmod F_\infty =\B_l$.  Actually, the part of the Chapoton family $F_{Chap(h)}$ for $h\in (10,\infty)$  does not belong to $P_3$ but to $\A_3$ and satisfies also Inductive Expectation.

5. The non-reduced cases may be able to handle either 1) extend Sturm theorem to handle non-reduced polynomials, or 2) every time we meet with non-reduced polynomial, then replace it by its reduction,  for which we apply again the above induction process (c.f.\ Discussion 13 in \S6).

\section{Zero loci of $F$-triangles: general case}

Based on the observations in previous sections, we formulate some conjectures on the zero loci of $F$-triangles on the unit square in the real $x$-$y$ plane.  It should be interesting and desirable to describe the behavior of zeros for all $A$-polynomials in $\A_l$ by giving a stratification of $\A_l$. However, our knowledge at present is limited, and  we restrict our attention only to the ``stratum" where our expectation should work.  In Conjecture A in the present section, we discuss some ``inductive structures on the stratum".  A self-contained formulation of the ``stratum" is given in Conjecture B in the Appendix.  These approach give another view on the zero loci for $f^+$-polynomials and $f$-polynomials studied in \cite{I-S} and \cite{I1}.

\bigskip
Conjecture A consists of 6 parts. The 6 parts are not independent but are logically overlapping or parallel and dependent to each other as we shall see in the discussions. But, we employed this rather verbose style, since we don't know yet the total logical structure to prove the conjecture.

\bigskip
\noindent
{\bf Conjecture A.}  There exists a semi-algebraic subset $\A_{l,0}$ of $\A_l$ for each $l=0,1,2,\cdots$, which satisfies the following A1., A2., A3., A4., A5. and A6.

\smallskip
A1.  The set $\A_{l,0}$ is non-empty for $l\in \Z_{\ge0}$. More precisely, $\A_{0,0}=\{1\}$, and for any finite Coxeter group (which may not necessarily be irreducible) of type $\Phi$ of rank $l\in \Z_{\ge1}$, the associated $F$-triangle $F_\Phi$ belongs to $\A_{l,0}$.

\smallskip
A2.  The set $\A_0:=\cup_{l=0}^{\infty} \A_{l,0}$ is closed under the product 
and the factorization in the set of real polynomials $F\in \R[x,y]$ with the normalization $F(0)=1$.

\smallskip
In the following A3., A4., A5. and A6.,  we formulate the statements of the conjecture for any fixed element $F\in \overline{P}_{l,0}$. 

\smallskip
A3.  The restriction $F|_y$ of $F$ to any fixed value $y\in [0,1]$, as a polynomial in $x$, has $l$ real roots on the interval $[0,1]$.

 \smallskip
We denote by $x_1(y), \cdots, x_l(y)$ the set of roots of $F|_y$ in its increasing order:
$$
0 \ \le \  x_1(y) \  \le \  \cdots \ \le \ x_l(y) \  \le \ 1 \ 
$$
such that $x_l(0)=1$ and $x_1(1)=0$. As a general fact, each $x_i(y)$ for $i=1,\cdots,l$ is a continuous function in $y\in[0,1]$ and is real analytic except at finite points. 

\medskip
\noindent
{\it Definition.} We call $y\in [0,1]$ is a {\it upper {\rm (resp}.\ lower) bending} point of the function $x_i$ ($1\le i\le l$), 
if there exists $1\le j\le l$ such that $i<j$ (resp.\ $i>j$) and $x_i(y)=x_j(y)$.

\smallskip
A4.  The $x_i$ ($i=1,\cdots,l$) as functions on $y\in (0,1)$ are monotone decreasing in a strong sense that the derivative 
{\large $\frac{dx_i}{dy} $} is negative at all non-bending points $y\in (0,1)$.

\smallskip
A5.  
 i)   If $F$ belongs to $\A_{l,0}$, then  $DF$ belongs to $\A_{l-1,0}$. 

\smallskip 
 ii)  For $F\in \A_{l,0}$ and $y\in [0,1]$, the zero-loci of $(DF)|_y$, say $x'_1(y), \cdots, x'_{l-1}(y)$ in increasing order, separate the zero-loci of $F|_y$.  That is, one has:
 $$
0 \ \le \  x_1(y) \  \le \ x'_1(y)\le  \ \cdots \ \le x'_{l-1}(y) \ \le \ x_l(y) \  \le \ 1 \ .
$$

iii) For $y\in (0,1)$, 
an equality $x_i(y)=x'_i(y)$ (resp.\ $x'_{i-1}(y)=x_{i}(y)$) holds if and only if $y$ is an upper (resp.\ lower) bending point of $x_i$.

\smallskip 
 A6.  Suppose that polynomials $F, DF,  D^2F,\cdots,  D^lF$ are reduced. Then,  
  for $y$ in a dense subset of $[0,1]$, the sequence of polynomials $F|_y, (DF)|_y,$ $(D^2F)|_y,$ $\cdots,  (D^lF)|_y=1$ in $x$ form a Sturm sequence for the interval $[0,1]$ in the following sense i) and ii) (c.f.\ \cite{T} \S16. See Discussions 10.,11.,12. and 13. below). 

\smallskip
i) For any $0<k<l$ and  root $x_0\in [0,1]$ of $(D^kF)|_y$, one has the inequality 
$$
D^{k-1}F(x_0,y) \ D^{k+1}F(x_0,y)>0
$$

ii) For all $1\le i \le l$, one has the inequality (see Footnote 8.)
$$
\partial_x F(x_{i}(y),y) \
DF(x_{i}(y),y) < 0.
$$

This completes the formulation of Conjecture A.

\bigskip
\noindent
{\bf Discussions on Conjecture A.}  

In the following, we  list up some discussions on the conjectures randomly.

\medskip
1.  A formal approach to Conjecture is the following. Since A3.-6.\ on the set of $F$ are semi-algebraic (note that $D:\A_l\to \A_{l-1}$ is a semi-algebraic map), one can inductively construct  semi-algebraic subset $\A_{l,0} \subset \A_{l}$ ($l\in\Z_{\ge3}$) satisfying A3.-6.\ (here $\A_{l,0}=\A_{l}$ for $l=0,1$). Then, we ask whether $\A_0:=\cup_{l=0}^\infty \A_{l,0}$ satisfies A1.\ and A2.  Here, presumably, $\A_{2,0}=P_2$, and $\A_{3,0}$ seems to be the domain in $\A_3$ bounded by the line  $\{F_{A_1 I_2(p)}\}_{p\in\R_{\ge2}}$ and the extension of the curve defining the domain III in \S5, 3. Rank 3 case, b) (see Figure 1.\ where $\A_{3,0}$ is indicated by light shaded part), which is a large extension of the domain III.

 \medskip
2.   In \S5, we have confirmed pictorially that A3-6. of Conjecture A hold for $F$-triangles of  rank $\le 4$ (i.e.\ of types $A_2,B_2,G_2,I_2(p),A_3,B_3,$ $H_3,$ $A_1I_2(p), A_4,B_4,$ $D_4,F_4,H_4,  A_1A_3,A_1B_3,$ $A_1H_3,I_2(p)I_2(q)$). In particular, up to bending points, one solid curve (=a graph of $x_i$) belongs to each connected component of $[0,1]\times[0,1] \setminus$dashed curve, and one dashed curve (a graph of $x_i'$) belongs to each connected component of  $[0,1]\times[0,1] \setminus$dotted curve, etc.

\medskip
3.  One may weaken A3., if the following question is true: 
For a polynomial $F\in \overline{P}_l$ and $y\in [0,1]$, any real solution of the polynomial equation $F |_y=0$ in $x$ lies in the interval $[0,1]$. Actually, we expect the following stronger statements are true:  $F \mid_{x\in \R_{<0},y\in[0,1]}>0$ and $(-1)^l F\mid_{x\in \R_{>1},y\in[0,1]}>0$.

\medskip
4. {\it Example.}  Let $F_{I_l(s_1,\cdots,s_{[l/2]})}$ be a $F$-triangle of virtual type $I_l(s_1,\cdots,s_{[l/2]})$ of rank $l$ \eqref{virtual}. Then, it real roots corresponds to the real factors $F_{A_1}$ (if $l$ is odd) and  $F_{I_2(2-\alpha)}$ for a real root $\alpha\in \R_{\le0}$ of the equation $T^k+\sum_{i=1}^ks_iT^{k-i}=0$ (recall {\bf Remark 3.4.}). 
In particular, it is totally real if and only if 
the type $I_l(s_1,\cdots,s_{[l/2]})$ decomposes into a product of types $I_2(p)$ for $p\in\R_{\ge2}$ and $A_1$.

\medskip
5.  In case of $F\in P_l$, one may weaken the condition in Conjecture  3. that $F|_y$ has real roots only for $y$ in a dense subset of $[0,1]$, since, in general for a real polynomial $F$, the condition to be totally  real (i.e.\ all roots are real) is a closed condition w.r.t. the coefficients of $F$. 

\medskip
6.  It is quite important to bend the function $x_i$ (or, to choose correct branch of the curve of the graph of $x_i$) at the bending point (= the multiple root of $F_y$) (see examples of types $A_1A_3$, $A_1B_3$ and $A_1H_3$ in Figure 7).  Usually the bending direction does not coincides with the direction of its natural analytic continuation. It is also interesting to observe that the bended curve may be deformed to a smooth curve by the semi-group $\Sigma_{[l/2]}$-action. See Examples of types $A_1A3+0.1F_\infty^2$, $A_1B3+0.1F_\infty^2$ and $A_1H3+0.1F_\infty^2$ in Figure 7.

\medskip
7.  {\it Question.} If $F_\Phi$  is a $F$-triangle for $\Phi\in \Z_{\ge0}[Cox]$ where $\Phi$ is a type of an irreducible Coxeter group, then is there no bending point? More generally, is it  true for such $F_\Phi$ where at least one direct summand of $\Phi$ is irreducible.

\medskip
8.  A4. is an immediate consequence of A5. Namely,  we have the identity: $F(x_i(y),y)=0$ for any $1\le i\le l$. For a non bending point $y$, we derivate the equality by $y$ and get the identity
$$
\frac{d x_i}{dy} (y) \ \partial_xF(x_i(y),y)+  \partial_y F(x_i(y),y)=0. 
$$
Since $y$ is a non-bending point of $x_i$ and $x_i(y)$  is a simple root of $F|_y$, we see that $\partial_xF(x_i(y),y)\not=0$, and, further more, since $x_i$ is the $i$th root counted from the left of the equation $F|_y=0$, we obtain $(-1)^i\partial_xF(x_i(y),y)>0$. On the other hand, the fact that $x_i(y)$ belongs to the ``$i$th component of $[0,1]\times[0,1]\setminus \text{Zeros}(DF)$'' means  $(-1)^i \partial_y F(x_i(y),y)>0$.
Both together implies $\frac{d x_i}{dy} (y) <0$.

\medskip
9. The ii) and iii) of  A5. is paraphrased geometrically as follows.

\smallskip
 Consider the components decomposition of $[0,1]\times[0,1]\setminus Zeros(DF)$. Then, up to bending points and the terminal points at $y=0,1$, the graph of the function $x_i$ is contained in the components bounded by the graphs of $x'_{i-1}$ and $x'_i$.  The point of the graph for $x_i$ at a upper (resp.\ lower) bending point lies on the graph for $x'_i$ (resp.\ $x_{i-1}'$).

\medskip
10.  In A6, usual formulation of a Sturm sequence in wider sense (c.f.\ \cite{T} \S16) asks two more conditions:
 a)  Any two neighboring polynomials $(D^kF)|_y$ and $(D^{k+1}F)|_y$ have no common zero at a point $x \in [0,1]$, and b) $(D^lF)|_y$ has constant sign on the interval $x\in[0,1]$.

We removed the conditions, since they are automatically satisfied. Namely, b) is trivial since $D^lF=1$, and a) except that $D^kF$ and $D^{k+1}F$ has a common polynomial factor,  $(D^kF)|_y$ and $(D^{k+1}F)|_y$ have common zero only at finite number of $y\in (0,1)$ so that we have just only to remove those points from the consideration. On the other hand, it is easy to see that $D^kF$ and $D^{k+1}F$ has common factor only when $D^kF$ has a nontrivial multiple factor. But it is not allowed due to the assumption in A6. that  $D^kF$ is a reduce polynomial.

\medskip
11.  We remark that the inequality ii) in A5. is exactly the property we used in the above 8. to show $\frac{d x_i}{dy} (y) <0$.  Thus, A3. is a consequence of A5.  

\medskip
12.  Sturm Theorem says that if $(D^kF)|_y$ ($k=0,\cdots,l$) is a Sturm sequence for the interval $[0,1]$, then $F|_y$ has $n(1)-n(0)$ number of roots on the interval $[0,1]$,
where $n(0)$ (resp.\ $n(1)$) is the number of sign changes in the sequence $D^kF(0,y)$ (resp.\ $D^kF(1,y)$) for $0\!\le\! k \!\le \! l$. \footnote{
In the standard formulation of Sturm Theorem, one should ask the positivity in the condition of 6 ii),  and then the number of roots of $F$ in $[0,1]$ is given by $n(0)-n(1)$.
}
On the other hand, recalling the formula \eqref{x=01},
we obtain 
\[
D^kF(0,y)= (1- y)^{l-k}
\quad \text{and} \quad
D^kF(1,y)= (-y)^{l-k}.
\]
This implies that $n(0)=0$ and $n(1)=l$, and hence $F_y$ for generic $y\in [0,1]$ (and hence for all $y\in[0,1]$) has $l=\rank(F)$ number of roots on the interval $[0,1]$. That is, A3. is a consequence of  A6.  On the other hand, it is also well-known that in such Sturm sequence, the roots of $D^kF$ are separated by the roots of $D^{k+1}F$. In particular, ii) and iii) of A5. should be a consequence of A6.

\medskip
13. In A6,  even the assumption that  $D^kF$ are reduced is not satisfied, the conclusion that $F|_y$ has $l$ real roots in the interval $[0,1]$ for all $y\in [0,1]$ is expected to be true either by a suitable generalization of Sturm theorem, or by replacing non-reduce $D^kF$ by its reduction. To show  this, one may need to formulate some careful induction process on the rank $l$, which seems a bit technical and we omitted it. 

\smallskip
End of discussions on Conjecture A.

\bigskip

\noindent
{\bf Relation of zero loci of a $F$-triangle with that of $f^+$- and $f$-polynomial.} 

\smallskip
Let us explain now the relationship between Conjecture A in the present note, and the results obtained in Ishibe-Saito \cite{I-S} and Ishibe \cite{I1}.  

Consider the Chapoton $F$-triangle $F_\Phi(x,y)$ associated with a finite Coxeter group of type $\Phi$. Then 
$f^+_\Phi(x):=F_\Phi(x,0)$ is the generating function of the \# of cones in the positive part $\Delta_+(\Phi)$ of the cluster fan of type $\Phi$, called  $f^+$-polynomial. The other specialization $f_\Phi(x)=F(x,x)$ is the generating function of the \# of cones of the whole Cluster fan $\Delta(\Phi)$, called $f$-polynomial (one is referred to \cite{Ar,At,A-B-W,B-W,C1,F-Z1,F-Z2,F-Z3,K} for information on Chapoton conjecture and its relation to the non-crossing partition lattice).  

In \cite{I-S}, it is shown that $f^+_\Phi$ has exactly $\rank(\Phi)$ number of simple real roots on the interval $(0,1]$, and in \cite{I1}, it is shown that $f_\Phi$ has exactly $\rank(\Phi)$ number of simple real roots on the interval $(0,1]$ and that the $i$th root of $f_\Phi$ is strictly smaller than the $i$th root of $f^+_\Phi$.  
We remark that those results are natural consequences of Conjecture A of the present note. Namely, the roots of $f^+_\Phi$ is exactly the intersection of zero loci of $F_\Phi$ with the bottom interval $[0,1]\times 0$, which are given by $x_i(0)$ ($i=1,\cdots,l$). The roots of $f_{\Phi}$  is given by the intersection of the diagonal line of the square $[0,1]\times[0,1]$ with the zero loci of $F_{\Phi}$, where the monotone decreasing property  of the functions $x_i(y)$ in A4.\ implies that the diagonal intersects with each graph of $x_i(y)$ for $i=1,\cdots,l$ exactly once at a place less than $x_i(0)$.
 
 \smallskip
 \noindent
 {\bf Concluding Remark.}

 The appearance of the approach in \cite{I-S} and \cite{I1} and that of present note seems rather different. Namely, in \cite{I-S} and \cite{I1}, we have used Rodrigues type expressions for the polynomials $f^+_\Phi(x)$ and $f_\Phi(x)$, which connect those polynomials to the theory of orthogonal polynomials where there are several well-established rich machineries available. On the other hand, the approach in the present note uses the extra variable $y$ and  the higher derivation sequence of $F(x,y)$ by $y$ should (conjecturally) give Sturm sequence as polynomials in $x$ parametrized by $y$. Therefore, the author would like to expect that there should be a reasonable (i.e.\ particular and not arbitrary)  one parameter $y$-deformation theory of  orthogonal polynomials, which should cover and combine both approaches.

\section{Appendix}

We show that the $A$-triangle part of a $F$-triangle  satisfies a system of inequalities \eqref{boundarycondition1}, and conjecturally satisfies further another system of inequalities \eqref{boundarycondition2}.  We formulate in Conjecture B a question whether the union of both systems of inequalities determine the semi-algebraic set $\A_0$ whose existence was assumed in \S 6 Conjecture A.

First, we recall some necessary notation in a slightly generalized form.

 Let $\Phi\in \R_{\ge0}[Cox]$ be a homogeneous element of rank $l\in \Z_{\ge0}$. The image of $\Phi$ by the map $F$ \eqref{generalftriangle} is denoted by $F_\Phi$ and is called a generalized $F$-triangle (in particular, the image of $\Phi\in Cox$ is called a $F$-triangle). Then, dividing $F_{\Phi}$ by $F_\infty=x(x-1)$ as a polynomial in $x$, we obtain a decomposition \eqref{Atriangle}
 $$
\leftline{  (4.1) \qquad\qquad\qquad   $ F_\Phi(x,y) = A_\Phi(x,y) F_\infty +\Tr(\Phi) \ \B_l(x,y)$ }
 $$
where $\B_l= ((\!-\!y)^l \!-\!(1\!-\!y)^l)x\!+\! (1\!-\!y)^l $ \eqref{Btriangle} and $A_\Phi$ is a polynomial of degree $l\!-\!2$ called the $A$-triangle part of $F_\Phi$ (recall Lemma 4.1 for some basic properties of $A$-triangles).  
Recall also the normalized derivation  \eqref{D-derivation}, 
$D :  F_\Phi \mapsto -\frac{1}{l}\partial_y F_\Phi$ for a (generalized) $F$-triangle $F$ of rank $l$ such that $DF_\Phi= F_{\partial \Phi/l}$ (Proposition 4.) is again a (generalized) $F$-triangle of rank $l-1$. Apply $D$ to the decomposition \eqref{Atriangle}. Noting the facts $D\B_l=\B_{l-1}$ and $DF_\infty=0$, we obtain
$$
D A_{\Phi} \ =\  A_{\partial \Phi/l}. 
$$

\medskip
The goal of the present Appendix is to show the following.
\begin{lem}
\label{boundarycondition}
{\rm 1.}  For any $\Phi\in \R[Cox]$ homogeneous of rank $l$ ($l\ge2$) and an integer $k$ with $0\le k\le l-2$, one has
\begin{equation}
\label{boundarycondition1}
(-1)^{l-k}\left(D^kA_{\Phi}|_{x=1,y=0} -Tr(\Phi) \right) \ \ge \ 0 .
\end{equation}
{\rm 2.} If Conjecture A in \S6 is true, then, for any $\Phi\in Cox$ of rank $l$ ($l\ge2$) and an integer $k$ with $0< k\le l-2$,  one has
\begin{equation}
\label{boundarycondition2}
(-1)^{[\frac{l-k}{2}]}\left(R(1-xD^{k-1}A_\Phi|_{y=0},1-xD^{k}A_\Phi|_{y=0} )\right)\ge0 
\quad \footnote{We denote by $R(P(x),Q(x))$  the resultant $c^nd^m \prod_{i=1}^m\prod_{j=1}^n(a_i-b_j)$ of the polynomials $P(x)\!=\!\!c\prod_{i+1}^m\!(x\!-\!a_i\!)$ and $Q(x)\!=\!\!d\prod_{j=1}^n\!(x\!-\!b_j\!)$ for constants $a_i,b_j$ and non-zero contants $c,d$.
}
\end{equation} 
\end{lem}
\begin{proof}  1. We prove this by induction on $l$, where, in the lowest rank case $l=2$ and $k=0$, the $A$-triangle of of type $I_2(p)$ ($p\ge2$) is a constant function $p-1 \ge 1$ (recall \S4) so that the condition \eqref{boundarycondition1} is satisfied. 

For a homogeneous $\Phi$ of rank $l\ge 3$, consider \eqref{boundarycondition1} for $k=1$. Since one has $Tr(\partial\Phi)= l\cdot Tr(\Phi)$ and $D\A_\Phi)=\A_{\partial\Phi/l}$, the equality \eqref{boundarycondition1} (by multiplying $l$)  is equivalent to $(-1)^{l-1}\big(A_{\partial\Phi} (1,0)-Tr(\partial\Phi) \big)  \ge  0$ where obviously $\partial\Phi$ is homogeneous of rank $l-1$. That is, the poof is reduced to the case $l-1$.  Similarly repeating the same type of calculations, the cases of  higher derivatives $D^k\A_\Phi$ are reduced to the lower rank cases. Thus, we have only to prove the case for $k=0$.

The LHS of \eqref{boundarycondition1} is $\R_{\ge0}$-linear with respect to $\Phi$. Therefore, we may reduce the proof to the cases when $\Phi$ is a monomial with coefficient $1$ in the sense that it consists of a single element $\Phi\in Cox$.

Suppose $\Phi$ is not irreducible but decomoses as $\Phi_1\Phi_2\cdots$. Then $F_{\Phi}$ also decomposes as $F_{\Phi_1}F_{\Phi_2}\cdots$ where each factor $F_{\Phi_i}$ ($i=1,2,\cdots$) has a root $x=1$ for the fixed $y=0$. That is, $x=1$ is a multiple root of the equation $F_{\Phi}(x,0)=0$, and hence one has $\partial_xF_\Phi(x,0) |_{x=1}=0$.  Applying $\partial_x$ on \eqref{Atriangle}, we see 
$$
0=\partial_xF_\Phi(x,0) |_{x=1}=
(\partial_xA_\Phi(x,0)\cdot F_\infty) |_{x=1} +(A_\Phi(x,y)\cdot \partial_xF_\infty)  |_{x=1}+ \partial_x\B_l(x,0) |_{x=1}
$$
where the first term vanishes, the second term is equal to  $A_\Phi(1,0)$ and the third term is equal to -1. Thus, we obtain $A_\Phi(1,0)=1$, a particular case of \eqref{boundarycondition1}.

Let $\Phi$ be a type of an irreducible Coxeter group. 
 We consider the specialization of \eqref{boundarycondition1} to $y=0$. Actually, the specialization of $F$-triangle in LHS is just the $f^+$-polynomial (the characteristic fucntion for the non-crossing partition of type $\Phi$), and it was also studied in \cite{I-S} under the name of {\it skew-growth function} $N_{G_\Phi^{dual +},\deg}(t)$ for the dual Artin monoid $G_\Phi^{dual+}$ of type $\Phi$ (c.f.\ \cite{S}).  That is 
$$
\begin{array}{c}
F_\Phi(x,0)= N_{G_\Phi^{dual +},\deg}(t) \quad (t=x).
\end{array}
$$
The quotient $\tilde{N}_{\Phi,\deg}(t)=N_{\Phi,\deg}(t)/(1-t)$ is called the reduced skew growth function [ibid]. Therefore, substituting $y=0$ (and $Tr(\Phi)=1$) in the decomposition formula \eqref{Atriangle}, we obtain the relation 
$$
\begin{array}{c}
(*)\qquad \qquad\qquad\qquad  \tilde{N}_{G_\Phi^{dual +},\deg}(t)=1-x\cdot A_\Phi(x,0) \qquad (t=x)\qquad \qquad \qquad 
\\
\end{array}
$$
between the reduced skew growth function and the $A$-triangle. In particular, the condition \eqref{boundarycondition} for $k=0$, is translated to the condition
$$ 
(-1)^l \tilde{N}_{\Phi,\deg}(1)\  \le\  0 .
$$
Using a trivial relation $ \tilde{N}_{G_\Phi^{dual+},\deg}(1)=-(\partial_t N_{G_\Phi^{dual+},\deg})|_{t=1}$, we need to caluclate the value of the partial derivative of $N_{G_\Phi^{dual +},\deg}$ at $t=1$. Actually, this was already done in \cite{I2} (cf. also Remark 4.4 of \cite{I-S}). So we have 
$$
\begin{array}{rcccccccccccccc}
Type: &A_l & B_l & D_l &\!\! E_6 &\!\!E_7 & E_8 &\! F_4 &\! G_2 &\!\!\! H_3 & H_4 & I_2(p)  \\
\!\!\!\!\!\tilde{N}_{\Phi,\deg}(1): &\!\!\!\!  (\!-\!1)^{l\!+\!1}    &\!\!\!\!  (\!-\!1)^{l\!+\!1}l   &\!\!\!\!  (\!-\!1)^{l\!+\!1}(l\!-\!2)   &\!\!\!  -7 &\!\!\! 16 &\!\!\! -44 &\!\!\!\! -10 &\!\!\! -4 &\!\!\! 8 &\!\!\! -42 &\! 2\!-\!p
 \end{array}
 $$   
 Clearly, this satisfies the required sign condition, and 1.\ of Lemma 7.1 is proven. 
 
 2. We use very little part of Conjecture A by specializing  to $y=0$. 
 
 Due to A1, $F_\Phi$ for any $\Phi\in Cox$ of rank $l$ belongs to $\A_{l,0}$, and A5. i) implies further $D^kF_{\Phi}$ belongs to $\A_{l-k,0}$ for $k=0,\cdots,l$.  
 
Next, applying A5. ii) to $D^{k-1}F_\Phi$ ($1\le k\le l-2$), we see the following system of inequalities:
$$
0 \ \le \  x_1(y) \  \le \ x'_1(y)\le  \ \cdots \ \le x'_{l-k-1}(y) \ \le \ x_{l-k}(y) \  \le \ 1 \ , 
 $$
 where $x_i(y)$ ($1\le i\le l-k+1$) (resp.\  $x_j'(y)$ ($1\le j\le l-k$) are the roots of $D^{k-1}F_\Phi|_y=0$ (resp. $D^{k}F_\Phi|_y=0$). Then, in view that $D^kF_\Phi|_y$ (resp.\ $D^{k+1}F_\Phi|_y$) are polynomials in $x$ of degree $l-k+1$ (resp.\ $l-k$) whose leading coefficients has the sign $(-1)^{l-k+1}$ (resp.\ $(-1)^{l-k}$), we see 
 $$
 (-1)^{\frac{(l-k+1)(l-k)}{2}} R[ D^{k-1}F_\Phi|_y,D^{k}F_\Phi|_y] \ \ge 0   
 $$
 for $y\in[0,1]$. If $y=0$, then the polynomials $D^{k-1}F_\Phi|_{y=0}$ and $D^{k}F_\Phi|_{y=0}$ 
 have the common factor $1-x$. Dividing by the factor, we still obtain
 $$
\!\!\!\!(**) \qquad\qquad \qquad  (-1)^{\frac{(l-k)(l-k-1)}{2}} R\left[ \frac{D^{k-1}F_\Phi|_{y=0}}{1-x},\frac{D^{k}F_\Phi|_{y=0}}{1-x}\right] \ \ge 0  .\qquad\qquad 
 $$
 On the other hand, applying the formula ($*$) in part 1.\ to $D^kF_\Phi$, we get  
 $$
 \frac{D^{k-1}F_\Phi|_{y=0}}{1-x} = 1-xD^{k-1}A_\Phi|_{y=0} \text{ \ \ and \ \ } \frac{D^{k}F_\Phi|_{y=0}}{1-x}=1-xD^{k}A_\Phi|_{y=0}.
 $$
 Substituting them to the previous formula $(**)$ and taking the sign calculation
 $ (-1)^{\frac{(l-k)(l-k-1)}{2}}=(-1)^{[\frac{l-k}{2}]}$ in account,  
  we obtain the relation \eqref{boundarycondition2}, and 2.\ of Lemma 7.1 is proven. 
 
  This completes a proof of Lemma 7.1.
 \end{proof}

\noindent
{\bf Problem.}  Prove 2. of Lemma 7.1 directly without assuming Conjecture A.

\medskip
\noindent
{\bf Example 1.}
  All $F$-triangles in Fig.\ 3,4,6,7,8, 9 and 10 satisfy (7.2).

\medskip
\noindent
{\bf Example 2.}
  The $F$-triangle $F_{I_l(s_1,\cdots,s_{[l/2]})}$ of virtual type $I_l(s_1,\cdots,s_{[l/2]})$ (recall \eqref{virtual}) satisfies (7.2) (since $1-x A_{I_l(s_1,\cdots,s_{[l/2]})}|_{y=0}$ has always the factor $1-x$ for $l>2$ so that the resultant (7.2)  is 0 for $k<l-2$). 

\medskip
\noindent
{\bf Example 3.}
 Let us describe (7.1) and (7.2) explicitly for rank  2 and 3 cases. 

\noindent
{\bf Rank 2} : The space of $A$-polynomials is given by $(s+1)F_\infty+\B_2$ ($s+1\in\R_{>0}$). Then there is only one condition \eqref{boundarycondition1} $(-)^{2-0}((s+1)-1)=s-1\ge0$, which was exactly the condition for $P_2$ (recall \S4, Rank 2 $A$-triangles).
 
\noindent
{\bf Rank 3}:  Recall that a $A$-polynomial $F\in \A_3$ of rank 3 has an expression
$
F=\big(a(1/2-x) +b(1/2-y)\big)F_\infty +\B_3
$
where $a,b\in \R_{>0}$ are arbitrally constants, and $A(x,y):=a(1/2-x) +b(1/2-y)$ is the $A$-triangle part of $F$. Then, 
$$
(7.1) \quad \Leftrightarrow \quad -a/2+b/2-1\le0   \quad \text{and }\quad b/3-1\ge0 ,\qquad\qquad\qquad\qquad\qquad
\vspace{-0.2cm}
$$
and
\vspace{-0.3cm}
$$
(7.2) \quad \Leftrightarrow \quad  R(1-x(a(1/2-x)+b/2), 1-bx)=(18a-3ab-b^2)/18\le 0.
$$
The first linear conditions are  exactly the conditions studied in Figure 1., and the second quadratic condition is exactly equivalent to the condition b) in \S5, Rank 3 case. See domain III in Figures 1 and 5. These together imply also that 

\smallskip
\noindent
{\bf Fact.} {\it The Chapoton family $\{F_{Chap(h)}\}_{h\in\R_{\ge1}}$ satisfies conditions {\rm (7.1)} and {\rm (7.2)}}.

\begin{rem}
As one sees immediately in the proof, that the conditions \eqref{boundarycondition1}  (resp.\ \eqref{boundarycondition2}) give general rules on the possible direction of the deformation of an $F$-triangle when $D^kF$ has a multiple root $x=1$ for $y=0$  (resp.\ when $D^{k-1}F$ and $D^kF$ have a common root $x\in(0,1)$ for $y=0$). Due to rotation invariance of the F-triagle \eqref{rotation}, the conditions also give general rules on the deformation of $F$ if some multiple roots phenomena occurs for $y=1$.  On the other hand, the conditions (7.1) and (7.2) are not sufficint  for a $A$-polynomial to be totally real (compare the above Example 2. with \S6 Discussion 4. Example).

 Therefore, the author would like to suspect that the conditions \eqref{boundarycondition1} and \eqref{boundarycondition2} together with the totally real condition on a $A$-polynomial $F$: 

\medskip 
$ \A_{l}^{tr}: =\{F\in \A_l\mid  F|_y$ for all $y\in [0,1]$ is a totally real polynomial in $x \}$
 
\medskip
\noindent
determine the conjectural domain $\A_0$ in \S6.  More explicitly, we ask
 
 \bigskip
 \noindent
 {\bf Conjecture B.} 
For  $l\in \Z_{\ge0}$, let  $\A_{l,0}$ be the set of all totally real $A$-polynomials $F\in  \A_i^{tr}$ of rank $l$ whose $A$-triangle part $A$ is satisfying 
$$
\begin{array}{cl}
\vspace{0.2cm}
(-1)^{l-k}\left(D^kA|_{x=1,y=0} -1 \right) \ \ge \ 0  &     (0\le k\le l-2), \\
(-1)^{[\frac{l-k}{2}]}\left(R(1-xD^{k-1}A|_{y=0},1-xD^{k}A|_{y=0} )\right)\ge0 & (0<k\le l-2).
\end{array}
$$
Then, $\A_0:=\cup_{l=0}^\infty \A_{l,0}$ satisfies Conjecture A.

\end{rem}


\bigskip
\emph{Acknowledgement.}~
The author express his gratitudes to 
Christos Athanasiadis, Frederic Chapoton, Christian Krattenthaler, and
Tadashi Ishibe for discussions and interests, and to Yoshihisa Obayashi for the helps in drawing Figures. 
A particular gratitudes goes to Frederic Chapoton, who, after looking at an early version of the present note, informed the author other examples of $F$-triangles which supported the author to formulate conjecture in terms of $A$-polynomials.

\medskip
\noindent
This research was supported by World Premier International Research Center Initiative (WPI Initiative), MEXT, Japan, by JSPS KAKENHI Grant Number 25247004, and by JSPS bilateral Japan - Russia Research Cooperative Program.

\end{document}